\documentclass[reqno]{amsart}
\usepackage{amssymb}
\usepackage{hyperref}
\usepackage{graphicx}
\usepackage{booktabs}
\usepackage{multirow}

\newtheorem{theorem}{Theorem}[section]
\newtheorem{proposition}[theorem]{Proposition}
\newtheorem{lemma}[theorem]{Lemma}

\theoremstyle{definition}

\theoremstyle{remark}
\newtheorem{remark}{Remark}[section]
\newtheorem*{note}{Note}

\numberwithin{equation}{section}

\begin{document}

\title[Elementary Askey-Wilson Functions]
{Elementary Askey-Wilson Functions}

\author{J.F. van Diejen}
\address{Instituto de Matem\'aticas, Universidad de Talca,
Casilla 747, Talca, Chile}

\email{diejen@utalca.cl}

\author{A.N. Kirillov}
\address{Yanqi Lake Beijing Institute of Mathematical Sciences and Applications, Huairou District, Beijing, China}

\email{kirillovxyz@gmail.com}
%\address{Research Institute for Mathematical Sciences,
%Kyoto University, Kitashirakawa Oiwake-cho, Sakyo-ku, Kyoto
%606-8502, Japan}

\thanks{This work was supported in part by the Agencia Nacional de Investigaci\'on y Desarrollo (ANID), through
FONDECYT Grant \# 1250427.}

%    General info

\date{September, 2026}

\begin{abstract}
We present a determinantal evaluation formula for the
very-well-poised ${}_8\Phi_7$ basic hypergeometric Askey-Wilson function in terms of elementary functions.
The formula in question is valid
for discrete values of the four permutation-symmetric Askey-Wilson parameters;
specifically, these consist of all quadruples of parameter values that are
given, up to a sign, by powers $q^k$ with $k$ a positive integer or
half-integer, in such a way that each of the four types (positive/negative
sign, integer/half-integer power) occurs exactly once.
The proof hinges on a product formula for an associated
Cauchy-type determinant, which is of independent interest and is established
here by elementary means using Krattenthaler's `identification of factors' method for determinant evaluations.
\end{abstract}

\maketitle
%\tableofcontents

\section{Introduction}\label{sec1}
It is well-known that the requirement that a system of orthogonal
polynomials be the eigenfunctions of a second-order differential-
or difference operator is quite rigid. Indeed, this eigenfunction
property leads to the celebrated (basic) hypergeometric families
known as the {\em classical orthogonal polynomials}. The
polynomials in question form a hierarchy, the most general member of
which is given by the {\em Askey-Wilson polynomials}
\cite{ask-wil:some,gas-rah:basic}; the other families of classical
(basic) hypergeometric orthogonal polynomials can be obtained from
these Askey-Wilson polynomials via parameter specializations and
limiting transitions \cite{koe-les-swa:hypergeometric}.

The Askey-Wilson polynomials are eigenfunctions of a second-order
($q$-)difference operator, commonly referred to as the
Askey-Wilson operator. By scaling the step size to zero, the
Askey-Wilson polynomials degenerate to the Jacobi polynomials and
the eigenvalue equation degenerates to the Gauss hypergeometric
equation. Besides these
polynomial solutions, the Gauss hypergeometric equation also
admits non-polynomial solutions given by the
celebrated Gauss hypergeometric
series. From the point of view of the eigenvalue problem, the
hypergeometric solutions correspond to generic values of the
spectral variable: at a discrete sequence of special spectral
values the hypergeometric series truncates and produces the Jacobi
polynomials.

This classical state of affairs has prompted the study of
non-polynomial eigenfunctions of the Askey-Wilson operator,
expressed in terms of very-well-poised ${}_8\Phi_7$ basic
hypergeometric series
\cite{ata-sus:difference,ism-rah:associated,rah:askey,sus:some1,sus:some2}.
The non-polynomial eigenfunctions under consideration are often
referred to as {\em Askey-Wilson functions}. For special values of the spectral
variable truncation again occurs, such that the Askey-Wilson
functions reduce to the Askey-Wilson polynomials. The study of the non-polynomial eigenfunctions of the
Askey-Wilson operator has been further stimulated by their
appearance in a range of applications involving the theory of
quantum integrable systems \cite{len:non-polynomial,rou:virasoro,rui:systems,rui:generalized}, the
harmonic analysis on noncompact quantum groups
\cite{bult:ruijsenaars,koe-sto:askey-wilson,sto:askey-wilson}, and the study of the
bispectral problem
\cite{gru-hai:some,hai-ili:askey-wilson}.

The main purpose of the present work is to provide an explicit
determinantal evaluation formula for an ${}_8\Phi_7$ basic
hypergeometric series representing the Askey-Wilson function at parameter
values given, up to a sign, by powers $q^k$ with $k$ a positive integer or
half-integer. For values of the spectral parameter in the polynomial spectrum, the corresponding
Askey-Wilson polynomials were identified in \cite[Section 8]{spi-zhe:discrete} as the family of all Askey-Wilson polynomials that can be retrieved from the Chebyshev polynomials of the second kind by means of a finite number of Darboux transformations.
Our determinantal formula arises by
making contact with the theory of the (infinite) Toda chain
\cite{die:sato,die:dynamics,fla:toda,fad-tak:hamiltonian,ges-hol-sim-zha:toda,tes:jacobi,tod:theory}. Specifically, we identify the pertinent Askey-Wilson
functions as particular instances of the solitonic
Baker-Akhiezer function for the Toda chain corresponding to particular sets of reflectionless
spectral data. An essential ingredient for establishing the link
between the Baker-Akhiezer function for the Toda chain on the one
hand, and the Askey-Wilson function at special (reflectionless) parameter values on
the other hand, hinges on a remarkable factorization of the tau
function for the Toda chain at the spectral data of
interest. By means of this factorization, it is seen that the Askey-Wilson operator becomes---at
the parameter values in question and upon an exponential change of
variable---a reflectionless Jacobi operator of soliton type, and the
corresponding Askey-Wilson function amounts to its Baker-Akhiezer function (which is
what accounts for the elementary nature of the evaluation formula). For a very special one-parameter subfamily of
Askey-Wilson functions and for the degenerate case of the Gauss hypergeometric
series, analogous determinantal evaluation formulas were found
previously in \cite{die-kir:formulas} and
\cite{die-kir:determinantal}, respectively (cf. also
\cite{gal:new,gal:wronskian,mat:functional-difference} for an alternative approach based on Wronskian- and Casorati-type determinants stemming from Darboux transformations).
Moreover, closely related parameter subfamilies of the Askey-Wilson functions arise in connection with
the study of the bispectral problem \cite{hai-ili:askey-wilson,ili:bispectral}.
The upshot is that for the complete discrete family of reflectionless parameter values indicated above,
our determinantal formulas below provide a closed expression for the Askey-Wilson function in terms of elementary functions.

The paper is organized as follows. Section \ref{sec2} collects some essential
preliminaries regarding the Askey-Wilson function. As our main result, Section
\ref{sec3} states the determinantal evaluation formulas for
the Askey-Wilson function 
 (cf. Theorems \ref{determinant:thm} and \ref{evaluation:thm}), together with the underlying
product formula for the tau function of the Toda chain
at the corresponding spectral data (cf. Theorem \ref{awc:thm}). The proof
of the determinantal- and evaluation formulas is given in Section
\ref{sec4}. This proof hinges on the product formula for the tau
function, whose proof is relegated to Section \ref{sec5}. Let us stress that the product formula in question
constitutes an evaluation of a Cauchy-type determinant that is of
independent interest. Its justification in Section \ref{sec5} is elementary and
self-contained, in the sense that it appeals neither to the theory of the
Askey-Wilson functions nor to that of the Toda chain. Specifically, the proof implements Krattenthaler's
`identification of factors' method \cite[Section 2.4]{kra:advanced} for computing the determinant;
it rests on a representation of the tau function as a ratio of alternants, combined with an
analysis of the zeros of the tau function based on a reflection symmetry of the
columns of the alternant.

\begin{note}{\em i)}
Throughout this paper we will use the conventions that empty sums vanish and that
empty products are equal to $1$.
\end{note}

\begin{note}{\em ii)}
We will use the following standard notations from the theory of
basic hypergeometric series \cite{gas-rah:basic}.
The $q$-shifted factorials are denoted by
\begin{equation*}
(a;q)_k:=\begin{cases}
1 &\text{for}\; k=0, \\
(1-a)(1-aq)\cdots (1-aq^{k-1}) &\text{for}\; k=1,2,3,\ldots ,
\end{cases}
\end{equation*}
with the convention that $(a;q)_\infty:=\prod_{n=0}^\infty
(1-aq^n)$ (for $|q|<1$). Products of $q$-shifted factorials are
abbreviated in the usual way via
\begin{equation*}
(a_1,\ldots ,a_{r};q)_k:=(a_1;q)_k\cdots (a_r;q)_k.
\end{equation*}

The ${}_{r+1}\Phi_r$
basic hypergeometric series is defined as
\begin{equation*}
{}_{r+1}\Phi_r \left[ \begin{array}{c} a_1,\ldots ,a_{r+1} \\
b_1,\ldots ,b_r  \end{array}  \mid q;z  \right] :=
\sum_{k=0}^\infty
\frac{(a_1,\ldots ,a_{r+1};q)_k}{(q,b_1,\ldots ,b_{r};q)_k} z^k ,
\end{equation*}
where it is assumed that the parameters are such that denominators
do not vanish and that $q,z$ lie inside the unit disc, so as to ensure
the convergence of the power series.
An important special case is formed by the ${}_{r+1}W_{r}$
very-well-poised basic
hypergeometric series
\begin{align*}
{}_{r+1}W_{r}(a_0;a_1,\ldots , a_{r-2}\mid q;z)  &:=
 {}_{r+1}\Phi_{r} \left[
\begin{array}{c} a_0,q\sqrt{a_0},-q\sqrt{a_0},a_1,\ldots ,a_{r-2} \\
\sqrt{a_0},-\sqrt{a_0},qa_0/a_1,\ldots ,qa_0/a_{r-2}  \end{array}  \mid q;z
 \right] \\
&=  \sum_{k=0}^\infty \frac{1-a_0q^{2k}}{1-a_0}
\frac{(a_0,a_1,\ldots ,a_{r-2};q)_k}{(q,qa_0/a_1,\ldots
,qa_0/a_{r-2};q)_k} z^k .
\end{align*}
\end{note}

\section{Preliminaries}\label{sec2}
Askey-Wilson functions have been thoroughly studied from diverse angles in
\cite{ata-sus:difference,bult:ruijsenaars,gru-hai:some,hai-ili:askey-wilson,ism-rah:associated,koe-sto:askey-wilson,len:non-polynomial,rah:askey,rou:virasoro,rui:systems,rui:generalized,rui:generalized2,rui:parameter,sto:askey-wilson,sus:some1,sus:some2}.
In this section we have collected the main properties of the
Askey-Wilson functions that are needed for our purposes. 

Let $0<q<1$ and let $a,b,c,d$ denote four real parameters such that
$abcd$ is positive. For convenience, we distinguish (real-valued) dual parameters
$\hat{a},\hat{b},\hat{c},\hat{d}$
related to $a,b,c,d$ via the transformation
\begin{equation}\label{dual-par}
\hat{a}=\sqrt{q^{-1}abcd},\;\; \hat{b}=ab/\hat{a},\;\;\hat{c}=ac/\hat{a},
\;\;\hat{d}=ad/\hat{a}.
\end{equation}
The eigenvalue equation for the Askey-Wilson operator is given by the
$q$-difference equation
\begin{subequations}
\begin{align}
 A(z)A (q^{-1}z^{-1})\psi(qz,\hat{z})-\left(A(z)+A(z^{-1})\right)\psi(z,\hat{z})+
\psi(q^{-1}z,\hat{z}) & \nonumber \\
 = (\hat{z}+\hat{z}^{-1}-\hat{a}-\hat{a}^{-1})\psi(z,\hat{z}) , & \label{aw-ev1}
\end{align}
with
\begin{equation}\label{aw-ev2}
A (z)=\frac{(1-az)(1-bz)(1-cz)(1-dz)}{\hat{a}(1-z^2)(1-qz^2)} .
\end{equation}
\end{subequations}
Here $z$ is the geometric variable, $\hat{z}$ is the spectral
parameter, and we have gauged our wave function $\psi (z,\hat{z})$
such that the coefficient of $\psi (q^{-1}z,\hat{z})$ in the
eigenvalue equation is equal to $1$. From \cite[Equation
(1.13)]{ism-rah:associated}, one readily deduces the following
solution to the eigenvalue equation \eqref{aw-ev1}, \eqref{aw-ev2}
\begin{subequations}
\begin{align}
\psi (z,\hat{z};a,b,c,d \mid q ) &= q^{\log_q(z)\log_q(\hat{z})}
\frac{(\hat{a}\hat{z},qaz\hat{z}/\hat{a},qbz\hat{z}/\hat{a},qcz\hat{z}/\hat{a},
qdz\hat{z}/\hat{a} ;q)_\infty}{(q\hat{z}^2,q^2z^2\hat{z}/\hat{a};q)_\infty}
  \nonumber \\
& \times
{}_8W_7(qz^2\hat{z}/\hat{a};qz/a,qz/b,qz/c,qz/d,q\hat{z}/\hat{a}
\mid q;\hat{a}\hat{z}), \label{awf}
\end{align}
for generic $z,\hat{z}$ such that $|\hat{a}\hat{z}|<1$. (Here $\log_q(z):=\log (z)/\log (q)$.)
We will
refer to $\psi (z,\hat{z};a,b,c,d \mid q )$ \eqref{awf} as the {\em
Askey-Wilson function}.
This Askey-Wilson function is (apart from the different choice
of the gauge) not
exactly the same as the one considered by Suslov
\cite{sus:some1,sus:some2} and by Koelink-Stokman \cite{koe-sto:askey-wilson},
which corresponds in fact to a linear combination of the solutions
$\psi (z,\hat{z};a,b,c,d \mid q )$ and $\psi (z,\hat{z}^{-1};a,b,c,d \mid q )$
with $q$-periodic coefficients.
In the terminology of Rahman
\cite{rah:askey}, the Askey-Wilson function employed by Suslov and
by Koelink-Stokman is of the {\em first kind}
(corresponding---in the gauge of \cite{koe-sto:askey-wilson,sus:some1,sus:some2}---
to a $z\to z^{-1}$ parity-invariant
solution of the eigenvalue
equation), whereas here we will be dealing rather with an Askey-Wilson
function of the {\em second kind} (corresponding to a (Jost) 
asymptotic plane wave solution of the eigenvalue equation).

Following \cite{koe-sto:askey-wilson,sus:some1,sus:some2}, we can extend the
Askey-Wilson function \eqref{awf} to the product of the plane wave
$q^{\log_q(z)\log_q(\hat{z})}$ and a meromorphic expression in
$z,\hat{z}$ through Bailey's three-term transformation \cite[Equation
(III.36)]{gas-rah:basic}:
\begin{align}
\psi (z,\hat{z};a,b,c,d \mid q ) &= q^{\log_q(z)\log_q(\hat{z})}
\frac{(az,bz,\hat{a}\hat{z},\hat{b}\hat{z},qcz\hat{z}/\hat{a},
qdz\hat{z}/\hat{a};q)_\infty}{(qz^2,q\hat{z}^2,ab/q;q)_\infty} \nonumber\\
& \times {}_4\Phi_3 \left[
\begin{array}{c} qz/a,qz/b,q\hat{z}/\hat{a},q\hat{z}/\hat{b} \\
q ^2/ab,qcz\hat{z}/\hat{a},qdz\hat{z}/\hat{a} \end{array}  \mid q;q\right]
\; + \nonumber \\
& q^{\log_q(z)\log_q(\hat{z})}
\frac{(qz/a,qz/b,q\hat{z}/\hat{a},q\hat{z}/\hat{b},q\hat{a}z\hat{z}/c,
q\hat{a}z\hat{z}/d;q)_\infty}{(qz^2,q\hat{z}^2,q/ab;q)_\infty}  \nonumber
\\
& \times {}_4\Phi_3 \left[
\begin{array}{c} az,bz,\hat{a}\hat{z},\hat{b}\hat{z} \\
ab,q\hat{a}z\hat{z}/c,q\hat{a}z\hat{z}/d \end{array}  \mid
q;q\right] . \label{awf2}
\end{align}
\end{subequations}

\begin{remark}\label{r21}
From the ${}_4\Phi_3$ representation \eqref{awf2} it is clear that the Askey-Wilson
function satisfies the duality symmetry
(cf. \cite{hai-ili:askey-wilson,ili:bispectral,koe-sto:askey-wilson,rui:systems,rui:generalized})
\begin{equation}\label{dual-sym}
\psi (z,\hat{z};a,b,c,d \mid q )=\psi
(\hat{z},z;\hat{a},\hat{b},\hat{c},\hat{d} \mid q) .
\end{equation}
This duality symmetry implies that the Askey-Wilson function $\psi
(z,\hat{z};a,b,c,d \mid q )$ satisfies a dual $q$-difference equation in
the spectral parameter $\hat{z}$ of the form in
\eqref{aw-ev1}, \eqref{aw-ev2}, but with coefficients in which $z$ and $a,b,c,d$ are interchanged with
$\hat{z}$ and $\hat{a},\hat{b},\hat{c},\hat{d}$, respectively. In
other words, the Askey-Wilson function constitutes a kernel solving the corresponding
bispectral problem \cite{gru-hai:some,hai-ili:askey-wilson}.
\end{remark}

\begin{remark}\label{r22}
It is manifest from the very-well-poised ${}_8\Phi_7$ representation \eqref{awf} that the Askey-Wilson function $\psi (z,\hat{z};a,b,c,d \mid q )$
is symmetric in the parameters $a$, $b$, $c$ and $d$.
As it turns out, our Askey-Wilson function also exhibits an additional
invariance with respect to parameter reflections of the type
$e\rightarrow q/e$ applied simultaneously to an even number of the parameters
$a$, $b$, $c$, $d$. To see this it is sufficient (in view of the permutation symmetry)
to check this invariance for say $c\rightarrow q/c$, $d\rightarrow q/d$;
in this case the invariance can be easily inferred by means
of the duality symmetry \eqref{dual-sym}, upon observing that in the dual picture
the pairwise parameter reflection at issue translates into a permutation of the
dual parameters
$\hat{a}\leftrightarrow \hat{b}$, $\hat{c}\leftrightarrow \hat{d}$.
The permutations and pairwise reflections generate an action of the Weyl group
$D_4$ in the parameter space. We thus conclude that the above Askey-Wilson function
$\psi (z,\hat{z};a,b,c,d \mid q )$ is invariant with respect to this $D_4$ action
(cf. \cite{rui:generalized2,rui:parameter}).
\end{remark}

\section{Elementary Askey-Wilson Functions}\label{sec3}
In this section we present our main result: a determinantal
evaluation formula for the Askey-Wilson functions in terms of elementary
functions. The formula is valid when each of the parameters $a,b,c,d$---in
which $\psi$ is symmetric---equals, up to a sign, a power $q^k$ with $k$ a
positive integer or half-integer, and the four types (positive/negative sign,
integer/half-integer power) each occur exactly \emph{once} (cf. \eqref{int-par} below).

For instance, for the following special choice of parameters $a=q$, $b=-q$,
$c=q^{1/2}$, $d=-q^{1/2}$, the difference equation
\eqref{aw-ev1}, \eqref{aw-ev2} has constant coefficients. Indeed,
one then has that $\hat{a}=q$, $A(z)A(q^{-1}z^{-1})=1$, and
$A(z)+A(z^{-1})=q+q^{-1}$, whence the difference
equation trivializes to $\psi (qz)+\psi
(q^{-1}z)=(\hat{z}+\hat{z}^{-1})\psi(z)$.
The Askey-Wilson function $\psi (z,\hat{z};a,b,c,d \mid q)$ should thus reduce
in this situation to a linear combination of the plane waves
$q^{\log_q(z)\log_q(\hat{z})}$ and $q^{-\log_q(z)\log_q(\hat{z})}$
(with possibly $z$-dependent $q$-periodic coefficients).
In fact, it is not difficult to see that
\begin{equation}\label{trivial}
\psi (q^x,q^{\hat{x}}; q, -q,
q^{1/2},-q^{1/2} \mid q)=q^{x\hat{x}}.
\end{equation}
Indeed, for $x\to+\infty$ the Askey-Wilson
function \eqref{trivial} behaves as the
plane wave $q^{x\hat{x}}$ (cf. \eqref{plane-wave} below);
by (temporarily) picking $\hat{x}$ on the imaginary axis, this reveals that in the
decomposition of the Askey-Wilson function
$\psi (q^x,q^{\hat{x}}; q, -q, q^{1/2},-q^{1/2} \mid q)$
in plane waves the coefficient of $q^{-x\hat{x}}$ vanishes and the
coefficient of $q^{x\hat{x}}$ is equal to $1$.

Our principal aim is now to present a generalization of this
reduction formula for the Askey-Wilson function involving
parameters of the form
\begin{subequations}
\begin{equation}\label{int-par}
a=q^{a_- +1},\quad b=-q^{a_+ +1},\quad c=q^{b_- +1/2},\quad
d=-q^{b_+ +1/2},
\end{equation}
where $a_\pm,b_\pm\in\mathbb{N}_0:=\{0,1,2,\ldots \}$,
and thus
\begin{equation}\label{dual-int-para}
\hat{a}=q^{\hat{a}_- +1},\quad \hat{b}=-q^{\hat{a}_+ +1},\quad
\hat{c}=q^{\hat{b}_- +1/2},\quad
\hat{d}=-q^{\hat{b}_+ +1/2},
\end{equation}
with
\begin{equation}\label{dual-int-parb}
\begin{array}{ll}
\hat{a}_-=(a_- +a_+ + b_- +b_+)/2, &
\hat{a}_+=(a_- +a_+ - b_- -b_+)/2,\\ [1ex]
\hat{b}_-=(a_- -a_+ + b_- -b_+)/2, &
\hat{b}_+=(a_- -a_+ - b_- +b_+)/2.
\end{array}
\end{equation}
\end{subequations}
The fact that the Askey-Wilson function indeed reduces to an elementary
function at parameter values of the form in
\eqref{int-par} can already be anticipated from a construction of the eigenfunctions
of the Askey-Wilson operator by means of
{\em Darboux transformations}
\cite{spi-zhe:discrete,hai-ili:askey-wilson} or by means of {\em shift operators}
\cite{kal-mil:symmetry,cha:macdonald} (see also \cite{rui:parameter}).
Indeed,
iterated application of a suitable sequence of shift operators on
the plane wave $q^{x\hat{x}}$ yields a construction of the
Askey-Wilson function for such discrete parameter values in terms of
elementary functions. In practice, however, already at small
integral values for $a_+$, $b_+$, $a_-$, $b_-$ the formulas
for the elementary Askey-Wilson functions thus obtained become
exceedingly cumbersome, and no closed expressions covering this entire family
of parameter values were available. Theorem
\ref{determinant:thm} and Theorem \ref{evaluation:thm} below
provide such closed formulas for the elementary Askey-Wilson
functions at issue; they cover the parameter family \eqref{int-par}
in its entirety and uniformly in $a_\pm$, $b_\pm$, rather than case by case.
The proof of these formulas is relegated to Section
\ref{sec4} and onwards.

To describe the formulas in question some notation is needed.
Let us define diagonal matrices $\mathbf{A}(\hat{x})$ and $\mathbf{B}(x)$ of the form
\begin{subequations}
\begin{equation}
\mathbf{A}(\hat{x}) =
\left[ \begin{array}{ll}
\mathbf{A}^+(\hat{x}) & \mathbf{0} \\
\mathbf{0} & \mathbf{A}^-(\hat{x})
\end{array} \right] ,
\qquad
\mathbf{B}(x) =
\left[ \begin{array}{ll}
\mathbf{B}^+(x) & \mathbf{0} \\
\mathbf{0} & \mathbf{B}^-(x)
\end{array} \right]  ,
\end{equation}
with
\begin{equation}
\mathbf{A}^+(\hat{x}) =\text{diag}\left[
\frac{1-q^{\hat{x}-j/2}}{1-q^{\hat{x}+j/2}} \right]_{j\in I^+} ,
\qquad
\mathbf{A}^- (\hat{x})=\text{diag}\left[
\frac{1+q^{\hat{x}-j/2}}{1+q^{\hat{x}+j/2}} \right]_{j\in I^-} ,
\end{equation}
\end{subequations}
and
\begin{subequations}
\begin{align}
\mathbf{B}^{+}(x) = & \\
& \text{diag} \left[ \epsilon^+_jq^{j(x+1/2)}(1-q^j)
\prod_{k\in I^+,\, k\neq j} \left| \frac{1-q^{(j+k)/2}}{q^{j/2}-q^{k/2}} \right|
\prod_{k\in I^-}  \frac{1+q^{(j+k)/2}}{q^{j/2}+q^{k/2}}
  \right]_{j\in I^+} , \nonumber
\end{align}
\begin{align}
\mathbf{B}^{-}(x) = & \\
&  \text{diag} \left[ \epsilon^-_jq^{j(x+1/2)}(1-q^j)
\prod_{k\in I^+}  \frac{1+q^{(j+k)/2}}{q^{j/2}+q^{k/2}}
\prod_{k\in I^-,\, k\neq j} \left|
\frac{1-q^{(j+k)/2}}{q^{j/2}-q^{k/2}} \right| \right]_{j\in I^-} .
\nonumber
\end{align}
\end{subequations}
Here $I^+$ and $I^-$ denote index sets of the form
\begin{subequations}
\begin{equation}
 I^\pm :=
 \{ 1,2,\ldots , |n^\pm_+ -n^\pm_-|,|n^\pm_+ -n^\pm_-|+2,\ldots,n^\pm_+ +n^\pm_-\}  ,
\end{equation}
parametrized by four nonnegative integers
$n^+_+$, $n^+_-$, $n^-_+$, $n^-_-$,
and $\epsilon^+_j$ ($j\in I^+$), $\epsilon^-_j$ ($j\in I^-$) are associated
sign configurations
\begin{equation}
\epsilon_j^\pm  =
\begin{cases}
(-1)^{\eta^\pm_j+j}& \text{if}\ n^\pm_+\geq n^\pm_-\\
(-1)^{\eta^\pm_j}&\text{if}\ n^\pm_+ < n^\pm_-
\end{cases}
\qquad (j\in I^\pm ),
\end{equation}
with $\eta^+_j$ and $\eta^-_j$ counting the position of $j$ in
respectively $I^+$ and $I^-$
\begin{equation}
\eta^\pm_j =
\begin{cases}
j &\text{for}\ j\leq | n^\pm_+-n^\pm_-  | \\
(j +| n^\pm_+-n^\pm_-  | )/2 &\text{for}\ j > | n^\pm_+-n^\pm_-  |
\end{cases} .
\end{equation}
\end{subequations}
Finally, we denote by $\mathbf{C}$ the Cauchy matrix
\begin{subequations}
\begin{equation}
\mathbf{C} =
\left[ \begin{array}{ll}
\mathbf{C}_{++} & \mathbf{C}_{+-} \\
\mathbf{C}_{-+} & \mathbf{C}_{--}
\end{array} \right] ,
\end{equation}
with blocks of the form
\begin{equation}
\begin{array}{ll}
\mathbf{C}_{++}=[ (1-q^{(j+k)/2})^{-1} ]_{j,k\in I^+} , &
\mathbf{C}_{+-}=[ (1+q^{(j+k)/2})^{-1} ]_{j\in I^+,k\in I^-} ,\\ [1.5ex]
\mathbf{C}_{-+}=[ (1+q^{(j+k)/2})^{-1} ]_{j\in I^- ,k\in I^+} , &
\mathbf{C}_{--}=[ (1-q^{(j+k)/2})^{-1} ]_{j,k\in I^-} .
\end{array}
\end{equation}
\end{subequations}

After these notational preliminaries we are now in a position to formulate
the main results of this paper.

\begin{theorem}[Determinantal Formula]\label{determinant:thm}
Let $\hat{x}\in i\mathbb{R}$ and let $x\in\mathbb{R}$ be generic in the sense that
$\det [\mathbf{1}+\mathbf{B}(x)\mathbf{C}]\neq 0$.
For $a_\pm, b_\pm\in\mathbb{N}_0$,
the Askey-Wilson function \eqref{awf}, \eqref{awf2} with parameters \eqref{int-par}
admits the following determinantal representation:
\begin{subequations}
\begin{equation}
\psi (q^x,q^{\hat{x}}; q^{a_- +1}, -q^{a_+ +1}, q^{b_- +1/2}, -q^{b_+ +1/2} \mid q)=
q^{x\hat{x}} \frac{\det [\mathbf{1} +\mathbf{A}(\hat{x}) \mathbf{B}(x)\mathbf{C}]}
                  {\det [\mathbf{1}+\mathbf{B}(x)\mathbf{C}]} ,
\end{equation}
where the parameters
$n^+_+$, $n^+_-$, $n^-_+$, $n^-_-$
are related to the parameters
$a_\pm$, $b_\pm$ via
\begin{equation}
n^+_\pm = a_\pm+b_\pm
\quad\text{and}\quad
n^-_\pm =
\begin{cases}
a_\pm-b_\pm ,& \text{if}\ \ a_\pm\geq b_\pm , \\
b_\pm-a_\pm -1 ,&
\text{if}\ \ a_\pm< b_\pm .
\end{cases}
\end{equation}
\end{subequations}
\end{theorem}

Expansion of the Cauchy determinants in Theorem \ref{determinant:thm}
leads us to the following explicit evaluation formula
for the elementary Askey-Wilson functions.

\begin{theorem}[Evaluation Formula]\label{evaluation:thm}
The explicit evaluation of the Askey-Wilson function from Theorem \ref{determinant:thm}
is given in terms of elementary functions by
\begin{subequations}
\begin{equation}
\psi (q^x,q^{\hat{x}}; q^{a_- +1}, -q^{a_+ +1}, q^{b_- +1/2},-q^{b_+ +1/2} \mid q)=
q^{x\hat{x}} \frac{\chi (x,\hat{x})}{\tau (x)} ,
\end{equation}
where
\begin{align}
\lefteqn{ \chi (x,\hat{x})=
\sum_{\begin{subarray}{c} J^+\subset I^+ \\ J^-\subset I^-  \end{subarray}}
\prod_{j\in J^+} \epsilon^+_j q^{jx}
\Bigl(\frac{q^{j/2}-q^{\hat{x}}}{1-q^{\hat{x}+j/2}}\Bigr)
\prod_{j\in J^-} \epsilon^-_j q^{jx}
\Bigl(\frac{q^{j/2}+q^{\hat{x}}}{1+q^{\hat{x}+j/2}} \Bigr) } &  \\
& \makebox[4em]{}\times\prod_{\begin{subarray}{c} j\in J^+,\, k\in I^+\setminus J^+ \\
 \text{or} \\j\in J^-,\, k\in I^-\setminus J^-  \end{subarray}}
 \left| \frac{1-q^{(j+k)/2}}{q^{j/2}-q^{k/2}} \right|
\prod_{\begin{subarray}{c} j\in J^+,\, k\in I^-\setminus J^- \\
\text{or} \\j\in J^-,\, k\in I^+\setminus J^+ \end{subarray}}
 \frac{1+q^{(j+k)/2}}{q^{j/2}+q^{k/2}} \nonumber
\end{align}
and
\begin{align}
 \tau (x)&=
\sum_{\begin{subarray}{c} J^+\subset I^+ \label{AWtauexp}
\\ J^-\subset I^-  \end{subarray}}
\prod_{j\in J^+} \epsilon^+_j q^{j(x+1/2)}
\prod_{j\in J^-} \epsilon^-_j q^{j(x+1/2)} \\
& \makebox[4em]{}\times
\prod_{\begin{subarray}{c} j\in J^+,\, k\in I^+\setminus J^+ \\
\text{or} \\j\in J^-,\, k\in I^-\setminus J^-  \end{subarray}}
 \left| \frac{1-q^{(j+k)/2}}{q^{j/2}-q^{k/2}} \right|
\prod_{\begin{subarray}{c} j\in J^+,\, k\in I^-\setminus J^- \\
\text{or} \\j\in J^-,\, k\in I^+\setminus J^+ \end{subarray}}
 \frac{1+q^{(j+k)/2}}{q^{j/2}+q^{k/2}}  \nonumber \\
 &=
 \prod_{1\leq\ell\leq a_+} (1+q^{x+\ell})^{a_+ +1-\ell}(1+q^{x+1-\ell})^{a_+ +1-\ell}
 \nonumber \\
&\times
\prod_{1\leq\ell\leq b_+} (1+q^{x+1/2+\ell})^{b_+ -\ell}(1+q^{x+3/2-\ell})^{b_+ +1-\ell}
\nonumber \\
&\times
\prod_{1\leq\ell\leq a_-} (1-q^{x+\ell})^{a_- +1-\ell}(1-q^{x+1-\ell})^{a_- +1-\ell}
\nonumber\\
&\times
\prod_{1\leq\ell\leq b_-} (1-q^{x+1/2+\ell})^{b_- -\ell}(1-q^{x+3/2-\ell})^{b_- +1-\ell} .
\label{AWtaufac}
\end{align}
\end{subequations}
\end{theorem}

In the above theorems
we have picked $x$ real and $\hat{x}$ imaginary, so as to guarantee
the convergence of the ${}_8W_7$ basic hypergeometric
representation for the Askey-Wilson function
 \eqref{awf}. However, both the determinantal formula and the evaluation formula
extend in fact to meromorphic identities
in $x$ and $\hat{x}$, between the ${}_4\Phi_3$ basic hypergeometric representation
of the Askey-Wilson function \eqref{awf2} on the one hand and the
elementary determinantal and evaluation expressions exhibited in the above
theorems on the other hand.

The equality stated in Theorem \ref{evaluation:thm} between the
{\em expansion} and the {\em factorization} of the Askey-Wilson-Cauchy
determinant $\tau(x)=\det [\mathbf{1}+\mathbf{B}(x)\mathbf{C}] $
is a consequence of the following product formula.

\begin{theorem}[Askey-Wilson-Cauchy Determinant]\label{awc:thm}
Let $n^+_+, n^+_-, n^-_+, n^-_-$ be nonnegative and integral. Then
\begin{subequations}
\begin{align}
 \det [\mathbf{1}+\mathbf{B}(x)\mathbf{C}]
 &=
 \prod_{1\leq\ell\leq a_+} (1+q^{x+\ell})^{a_+ +1-\ell}(1+q^{x+1-\ell})^{a_+ +1-\ell}
 \label{tau-prod} \\
&\times
\prod_{1\leq\ell\leq b_+} (1+q^{x+1/2+\ell})^{b_+ -\ell}(1+q^{x+3/2-\ell})^{b_+ +1-\ell}
\nonumber \\
&\times
\prod_{1\leq\ell\leq a_-} (1-q^{x+\ell})^{a_- +1-\ell}(1-q^{x+1-\ell})^{a_- +1-\ell}
\nonumber\\
&\times
\prod_{1\leq\ell\leq b_-} (1-q^{x+1/2+\ell})^{b_- -\ell}(1-q^{x+3/2-\ell})^{b_- +1-\ell} ,
\nonumber
\end{align}
with
\begin{align}
&
\begin{cases}
a_\pm=(n^+_\pm+n^-_\pm)/2 \\
b_\pm=|n^+_\pm-n^-_\pm|/2 \\
\end{cases} ,\qquad\quad
\text{if}\quad n^+_\pm+n^-_\pm\ \text{even},
 \nonumber\\
&
\begin{cases}
a_\pm=(|n^+_\pm-n^-_\pm |-1)/2 \\
b_\pm=(n^+_\pm+n^-_\pm+1)/2 \\
\end{cases},
\quad \text{if}\quad n^+_\pm+n^-_\pm\ \text{odd} . \label{par-rel}
\end{align}
\end{subequations}
\end{theorem}

\begin{remark}
It is important to emphasize that Theorems \ref{determinant:thm} and \ref{evaluation:thm}
amount to the following explicit evaluation formula for the very-well-poised ${}_8\Phi_7$
basic hypergeometric series under consideration
\begin{subequations}
\begin{align}
& {}_8W_7 (q^{-\hat{a}_- +2x +\hat{x}}; q^{-a_- +x}, -q^{-a_+ +x},
         q^{\frac{1}{2}-b_- +x}, -q^{\frac{1}{2}-b_+ +x}, q^{-\hat{a}_- +\hat{x}}
         \mid q; q^{1+\hat{a}_- +\hat{x}})  \nonumber \\
& = c(x,\hat{x}) \frac{\det [\mathbf{1} +\mathbf{A}(\hat{x}) \mathbf{B}(x)\mathbf{C}]}
                  {\det [\mathbf{1}+\mathbf{B}(x)\mathbf{C}]} \nonumber \\
& = c(x,\hat{x}) \frac{\chi (x,\hat{x})}{\tau (x)} ,
\end{align}
with
\begin{align}
\lefteqn{ c(x,\hat{x}) =} & \\
& \!\!\!\!\!\!\!\!\!\!\frac{(q^{1+2\hat{x}}, q^{1-\hat{a}_- +2x+\hat{x}} ;q)_\infty}
     {(q^{1+\hat{a}_- +\hat{x}}, q^{1+a_- -\hat{a}_- +x+\hat{x}},
       -q^{1+a_+ -\hat{a}_- +x+\hat{x}},
        q^{\frac{1}{2}+b_- -\hat{a}_- +x+\hat{x}},
           -q^{\frac{1}{2}+b_+ -\hat{a}_- +x+\hat{x}} ;q)_\infty}  . \nonumber
\end{align}
\end{subequations}
In particular, for $a_-,a_+, b_-, b_+=0$ the determinantal factors become trivial
and the above formula simplifies to the product form
\begin{equation}\label{trivial-8W7}
 {}_8W_7 (q^{2x +\hat{x}}; q^{x}, -q^{x},
         q^{\frac{1}{2}+x}, -q^{\frac{1}{2} +x}, q^{\hat{x}}
         \mid q; q^{1+\hat{x}})   =
\frac{(q^{1+2\hat{x}},q^{1+2x+\hat{x}};q)_\infty}
     {(q^{1+\hat{x}},q^{1+2x+2\hat{x}};q)_\infty} .
\end{equation}
The latter summation formula can be independently recovered from the following continuous $q$-ultraspherical
reduction of the very-well-poised ${}_8\Phi_7$ series from \cite[Equation (3.2)]{rah:askey}:
\begin{align}\label{rahman}
& {}_8W_7(z^2q^{\lambda+1}; qz/a, q^{\frac12}z/a, -qz/a, -q^{\frac12}z/a, q^{\lambda+1}
 \mid q; a^4q^{\lambda}) \nonumber \\
& = \frac{(z^2q^{\lambda+2}, a^4q^{2\lambda+1};q)_\infty}
         {(a^2q^{\lambda+1}, a^2z^2q^{2\lambda+2};q)_\infty}\;
  {}_2\Phi_1 \left[ \begin{array}{c} qz^2/a^2, q/a^2 \\ qz^2 \end{array} \mid q; a^4q^\lambda \right] .
\end{align}
Indeed, upon substituting
$a=q^{\frac12}$, $z=q^{x}$, and $q^{\lambda+1}=q^{\hat{x}}$
into \eqref{rahman}, the leading parameter, the five permutation-symmetric parameters, and the argument of the ${}_8W_7$ series
become, respectively,
$q^{2x+\hat{x}}$; $q^{\frac12+x}$, $q^{x}$, $-q^{\frac12+x}$, $-q^{x}$, $q^{\hat{x}}$; $q^{1+\hat{x}}$,
so the left-hand side coincides with the ${}_8W_7$ series in \eqref{trivial-8W7}.
Moreover, on the right-hand side of \eqref{rahman} the numerator parameter
$q/a^2$ of the ${}_2\Phi_1$ series becomes equal to $1$, whence
the latter series reduces to its
zeroth term $1$ while the prefactor becomes
$(q^{1+2x+\hat{x}},q^{1+2\hat{x}};q)_\infty/(q^{1+\hat{x}},q^{1+2x+2\hat{x}};q)_\infty$.
\end{remark}

\begin{remark}
In view of the permutation symmetry, Theorems \ref{determinant:thm} and \ref{evaluation:thm} provide evaluation formulas for the Askey-Wilson function
$\psi (q^x,q^{\hat{x}};a,b,c,d \mid q)$ for parameters such that $ \{a,b,c,d\}=\{q^{1+a_-},-q^{1+a_+},q^{\frac{1}{2}+b_-},-q^{\frac{1}{2}+b_+}\} $ with $a_\pm,b_\pm\in\mathbb{N}_0$.
From the $D_4$ parameter symmetry described in Remark \ref{r22} at the end of
Section \ref{sec2}, however, it follows that these two evaluation formulas cover in fact $\psi (q^x,q^{\hat{x}};a,b,c,d \mid q)$ for parameters in the orbit with respect to the action of the $D_4$ reflection group, i.e. for
parameters of the form
\begin{equation}\label{D4-orbit}
  \{a,b,c,d\}
=
\bigl\{ q^{\frac{1}{2}+\varepsilon_1(\frac{1}{2}+a_-)},\,
        -q^{\frac{1}{2}+\varepsilon_2(\frac{1}{2}+a_+)},\,
         q^{\frac{1}{2}+\varepsilon_3 b_-},\,
        -q^{\frac{1}{2}+\varepsilon_4 b_+} \bigr\} ,
\end{equation}
for all $\varepsilon_1,\varepsilon_2,\varepsilon_3,\varepsilon_4\in \{ 1,-1\}$ such that $\varepsilon_1\varepsilon_2\varepsilon_3\varepsilon_4=1$. In other words,
this incorporates the situation that some of the parameters are of the form $\pm q^{k}$ with $k$ a nonpositive integer or half-integer. Specifically,
the formula covers all quadruples of parameter values that are
given, up to a sign, by powers $q^k$ with $k$ an integer or
half-integer, in such a way that each of the four types (positive/negative
sign, integer/half-integer power) occurs exactly once and the number of nonpositive powers is even if $\{a,b,c,d\}\cap \{ q^{1/2},-q^{1/2}\}=\emptyset$.
(Notice that the reflection of a parameter $\pm q^{1/2}$ is trivial, so the parity restriction disappears when such a parameter is present.)
Moreover, from the duality symmetry \eqref{dual-sym} it is clear that through the interchange $x\leftrightarrow \hat{x}$ in (the meromorphic extension of) the formulas of
Theorems \ref{determinant:thm} and \ref{evaluation:thm}, one obtains corresponding evaluation formulas for the
Askey-Wilson function $\psi (q^x,q^{\hat{x}};a,b,c,d \mid q)$ that are valid whenever the dual parameters $\hat{a},\hat{b},\hat{c},\hat{d}$ \eqref{dual-par} are given by a
quadruple of the form just described.
\end{remark}

\section{Proof of the Determinantal Evaluation Formula}\label{sec4}
The proof of the determinantal evaluation formula for the
Askey-Wilson function draws on results from the theory of
integrable systems. In a nutshell, the idea behind the proof is to
identify the elementary Askey-Wilson function for discrete
parameter values of the form \eqref{int-par} with a stationary
solitonic Baker-Akhiezer function for the Toda chain corresponding
to specific reflectionless spectral data. The pertinent results from
Toda theory needed for our purposes are contained in the papers
\cite{die:sato,die:dynamics,ges-hol-sim-zha:toda}; further
relevant background material concerning the Toda chain can be
found in the (standard) texts
\cite{fla:toda,fad-tak:hamiltonian,tes:jacobi,tod:theory}.

Our solitonic Baker-Akhiezer function will be parametrized by $N$ \emph{distinct} points
$\hat{z}_1,\ldots ,\hat{z}_N\in (-1,1)\setminus \{0\}$ and $N$ (not necessarily distinct) auxiliary parameters 
 $\hat{\nu}_1,\ldots ,\hat{\nu}_N\in\mathbb{R}\setminus\{0\}$. We will refer to these $2N$ parameters as \emph{the spectral data of the Baker-Akhiezer function}. 
Specifically, let us define the matrices
\begin{equation}
\mathbf{\hat{Z}}=\text{diag}(\hat{z}_1,\ldots ,\hat{z}_N) \quad\text{and}\quad
\mathbf{\hat{N}}=\left[  \frac{\hat{\nu}_j}{1-\hat{z}_j\hat{z}_k}\right]_{1\leq j,k\leq N} ,
\end{equation}
and employ the notational convention
$|\mathbf{\hat{Z}}|:=\text{diag}(|\hat{z}_1|,\ldots ,|\hat{z}_N|)$.
From the matrices $\mathbf{\hat{Z}}$, $\mathbf{\hat{N}}$ and the identity matrix $\mathbf{1}$, we now construct the
(stationary solitonic) Baker-Akhiezer wave function (of Toda type) $\psi_N(x,\hat{z})$ in the variables $x\in\{ x\in\mathbb{R}\mid \tau_N(x)\neq 0\}$ and $\hat{z}\in\mathbb{C}\setminus\bigl((-\infty,0]\cup\{\hat{z}_1^{-1},\ldots,\hat{z}_N^{-1}\}\bigr)$ as follows:
\begin{equation}\label{baf1}
\psi_N(x,\hat{z}) =  \frac{\hat{z}^x \chi_N(x,\hat{z})}{\tau_N(x)} ,
\end{equation}
where $\hat{z}^x$ is defined by means of the principal branch,
\begin{subequations}
\begin{align}
\chi_N(x,\hat{z}) &=   \det [ \mathbf{1}+(\mathbf{1}-\hat{z}\mathbf{\hat{Z}}^{-1})(\mathbf{1}-\hat{z}\mathbf{\hat{Z}})^{-1}
                                     |\mathbf{\hat{Z}}|^{2x+2}\mathbf{\hat{N}}  ]  \label{baf2a} \\
            &= \sum_{J\subset\{ 1,\ldots ,N\}}
       \prod_{j\in J} \Bigl( \frac{1-\hat{z}\hat{z}_j^{-1}}{1-\hat{z}\hat{z}_j} \Bigr)
       \Bigl(\frac{\hat{\nu}_j|\hat{z}_j|^{2x+2}}{1-\hat{z}_j^2}\Bigr)
       \prod_{\begin{subarray}{c}j,k\in J\\ j<k\end{subarray}}
       \Bigl(\frac{\hat{z}_j-\hat{z}_k}{1-\hat{z}_j\hat{z}_k}\Bigr)^2 ,
\label{baf2}
\end{align}
and
\begin{align}
\tau_N(x)  &=   \det [ \mathbf{1} + |\mathbf{\hat{Z}}|^{2x+2}\mathbf{\hat{N}} ]   \\
           &= \sum_{J\subset\{ 1,\ldots ,N\}}
       \prod_{j\in J} \frac{\hat{\nu}_j|\hat{z}_j|^{2x+2}}{1-\hat{z}_j^2}
       \prod_{\begin{subarray}{c}j,k\in J\\ j<k\end{subarray}}
       \Bigl(\frac{\hat{z}_j-\hat{z}_k}{1-\hat{z}_j\hat{z}_k}\Bigr)^2 .
\label{baf3}
\end{align}
\end{subequations}
In the above formulas, the explicit expansion of the determinants $\chi_N(x,\hat{z})$
and $\tau_N(x)$
hinges on the observation that a determinant of the form $\det (\mathbf{1}+\mathbf{M})$
amounts to the sum of the principal minors of $\mathbf{M}$;
for the matrices under consideration,
these minors are readily evaluated in closed form by means of the celebrated
Cauchy determinant formula
\begin{equation}
\det \left[ \frac{1}{1-x_jy_k} \right]_{1\leq j,k\leq N} =
\frac{\prod_{1\leq j<k\leq N} (x_j-x_k)(y_j-y_k)}{\prod_{1\leq j,k\leq N} (1-x_jy_k)}.
\end{equation}
As a function of the spectral variable $\hat{z}$, the
Baker-Akhiezer function has at most simple poles at $\hat{z}=\hat{z}_j^{-1}$, $j=1,\ldots ,N$, while as a function of the geometric variable $x$ the function is smooth on $\mathbb{R}$, provided the auxiliary parameters $\hat{\nu}_1,\ldots, \hat{\nu}_N$ are all positive. Moreover, in this situation $\hat{\nu}_j$ encodes the reciprocal of the squared norm of $\psi_N(x,\hat{z}_j)$, cf. e.g. \cite[Section VII.B]{die:sato}. Here, however, we allow
that one or more of the parameters $\hat{\nu}_j$ become negative. As a result, our Baker-Akhiezer function $\psi_N(x,\hat{z})$ is not necessarily smooth on $\mathbb{R}$ as a function of $x$,
due to the zero locus of the denominator $\tau_N(x)$.

Let us from now on pick parameters $\hat{\nu}_1,\ldots,\hat{\nu}_N$ of the form
\begin{equation}\label{parity-sym}
\hat{\nu}_j = \hat{\epsilon}_j |\hat{z}_j|^{-1}(1-\hat{z}_j^2)
\prod_{\begin{subarray}{c} 1\leq k\leq N\\ k\neq j\end{subarray}}
\left| \frac{1-\hat{z}_j\hat{z}_k}{\hat{z}_j-\hat{z}_k} \right|
\qquad (j=1,\ldots ,N),
\end{equation}
with $\hat{\epsilon}_1,\ldots ,\hat{\epsilon}_N\in\{ 1,-1\}$. The explicit expansions
for the numerator and denominator of the Baker-Akhiezer wave function
$\psi_N(x,\hat{z})$ \eqref{baf1} then become
\begin{subequations}
\begin{align}
\chi_N(x,\hat{z})  &= \sum_{J\subset\{ 1,\ldots ,N\}}
       \prod_{j\in J} \hat{\epsilon}_j|\hat{z}_j|^{2x+1}
\Bigl( \frac{1-\hat{z}\hat{z}_j^{-1}}{1-\hat{z}\hat{z}_j}\Bigr)
       \prod_{\begin{subarray}{c}j\in J\\ k\in J^c\end{subarray}}
       \Bigl|\frac{1-\hat{z}_j\hat{z}_k}{\hat{z}_j-\hat{z}_k}\Bigr|  ,\label{bafs1} \\
\tau_N(x)  &= \sum_{J\subset\{ 1,\ldots ,N\}}
       \prod_{j\in J} \hat{\epsilon}_j|\hat{z}_j|^{2x+1}
       \prod_{\begin{subarray}{c}j\in J\\ k\in J^c\end{subarray}}
       \Bigl|\frac{1-\hat{z}_j\hat{z}_k}{\hat{z}_j-\hat{z}_k}\Bigr| , \label{bafs2}
\end{align}
\end{subequations}
respectively (where $J^c:=\{ 1,\ldots ,N\}\setminus J $).

It follows from
\cite[Corollary 2, Theorem 3]{die:sato} and \cite[Theorem 1, Equation
(7.13)]{die:dynamics}, combined with \cite[Equations
(3.38),(3.40)]{ges-hol-sim-zha:toda}, that the Baker-Akhiezer wave function
$\psi_N(x,\hat{z})$ \eqref{baf1} with $\chi_N(x,\hat{z})$ and
$\tau_N(x)$ given by \eqref{bafs1} and \eqref{bafs2},
respectively, solves the second-order difference equation
\begin{align}
A_N(x)B_N(x+1)\psi_N(x+1,\hat{z})-\bigl(A_N(x)+B_N(x)\bigr)
\psi_N(x,\hat{z})+\psi_N(x-1,\hat{z}) &  \nonumber \\
 = (\hat{z}+\hat{z}^{-1}-\hat{z}_N-\hat{z}_N^{-1})\psi_N(x,\hat{z})
, & \label{deq1}
\end{align}
with
\begin{subequations}
\begin{align}
A_N(x) &= \hat{z}_N
\frac{\tau_N(x-1)\tau_{N-1}(x+1/2)}{\tau_N(x)\tau_{N-1}(x-1/2)}, \label{deq2}
\\
B_N(x)&= \hat{z}_N^{-1}
\frac{\tau_N(x)\tau_{N-1}(x-3/2)}{\tau_N(x-1)\tau_{N-1}(x-1/2)}
.\label{deq3}
\end{align}
\end{subequations}
The tau function $\tau_N(x)$ \eqref{bafs2} and the difference
equation from \eqref{deq1}, \eqref{deq2}, \eqref{deq3} enjoy
the parity symmetry
\begin{equation}\label{parity}
\tau_N(-x)=  \hat{\epsilon}_1\cdots\hat{\epsilon}_N |\hat{z}_1\cdots
\hat{z}_N|^{1-2x}\tau_N(x-1)\quad\text{and}\quad A_N(-x)=B_N(x).
\end{equation}
Hence, the difference equation under consideration has the same
structure as the Askey-Wilson difference equation
\eqref{aw-ev1}, \eqref{aw-ev2} (up to an exponential change of
variable).
In fact we claim that, for $N=\max(n^+_+,n^+_-)+\max(n^-_+,n^-_-)$ and
\begin{equation}\label{AWdata}
\{ (\hat{z}_1,\hat{\epsilon}_1),\ldots ,(\hat{z}_N,\hat{\epsilon}_N)\} =
\{ (q^{j/2},\epsilon_j^+)\}_{j\in I^+}\cup
\{ (-q^{j/2},\epsilon_j^-)\}_{j\in I^-} ,
\end{equation}
the difference equation in
\eqref{deq1}, \eqref{deq2}, \eqref{deq3} for the Baker-Akhiezer wave function
$\psi_N(x,\hat{z})$ \eqref{baf1}, \eqref{bafs1}, \eqref{bafs2}
coincides with
the Askey-Wilson difference equation from \eqref{aw-ev1}, \eqref{aw-ev2} with
$z=q^x$ and
parameters of the form in \eqref{int-par}.

Before proving this claim, let us first observe that Theorem \ref{determinant:thm} and
Theorem \ref{evaluation:thm} arise as immediate consequences.
Indeed, it is manifest that
for the spectral data in question the Baker-Akhiezer wave function
$\psi_N(x,q^{\hat{x}})$
reduces to the wave functions stated
on the right-hand-sides of the determinantal evaluation formulas in
Theorem \ref{determinant:thm} and Theorem \ref{evaluation:thm}, respectively.
In other words, our claim
implies that these right-hand-side functions satisfy the Askey-Wilson difference equation
(with $z=q^x$, $\hat{z}=q^{\hat{x}}$ and
parameters of the form in \eqref{int-par}).
That this solution of the Askey-Wilson difference equation in fact
coincides with the (basic hypergeometric) Askey-Wilson function
$\psi (q^x,q^{\hat{x}};q^{a_- +1},-q^{a_+ +1},q^{b_- +1/2},-q^{b_+ +1/2} \mid q)$
on the left-hand-sides of the stated formulas is then clear from the fact that
both solutions have the same plane-wave asymptotics of the form
$q^{x\hat{x}}$ for $x\to +\infty$. (Notice in this connection that for large $x$ the
difference equation degenerates to the trivial equation
$\psi (x+1)+\psi (x-1)=(q^{\hat{x}}+q^{-\hat{x}})\psi(x)$, the solutions of
which consist of a linear combination of the plane waves
$q^{x\hat{x}}$ and $q^{-x\hat{x}}$
characterized by possibly $x$-dependent $1$-periodic coefficients.)
For the Baker-Akhiezer function the
plane-wave asymptotics of the form $q^{x\hat{x}}$
is immediate from the explicit expressions, whereas
for the basic hypergeometric Askey-Wilson function
we read-off from the ${}_8W_7$-series representation \eqref{awf} that
\begin{equation}\label{plane-wave}
\psi (q^x,q^{\hat{x}};a,b,c,d \mid q)\stackrel{x\to +\infty}{\longrightarrow}
q^{x\hat{x}}
\frac{(\hat{a}q^{\hat{x}};q)_\infty}{(q^{1+2\hat{x}};q)_\infty}\;
{}_1\Phi_0\left[ \begin{array}{c} \hat{a}^{-1}q^{1+\hat{x}} \\ - \end{array}
\mid q;\hat{a}q^{\hat{x}} \right] = q^{x\hat{x}}
\end{equation}
(by the $q$-binomial formula \cite[Equation (II.3)]{gas-rah:basic}).

Now, to infer the claim that---with the above specialization of the
spectral data---the difference
equation for the Baker-Akhiezer function indeed coincides with the
Askey-Wilson eigenvalue equation, we will employ the factorization of the tau function
stated in Theorem \ref{awc:thm}:
\begin{align}
 \tau_{N} (x)
 &=
 \prod_{1\leq\ell\leq a_+} (1+q^{x+\ell})^{a_+ +1-\ell}(1+q^{x+1-\ell})^{a_+ +1-\ell}
 \nonumber \\
&\times
\prod_{1\leq\ell\leq b_+} (1+q^{x+1/2+\ell})^{b_+ -\ell}(1+q^{x+3/2-\ell})^{b_+ +1-\ell}
\nonumber \\
&\times
\prod_{1\leq\ell\leq a_-} (1-q^{x+\ell})^{a_- +1-\ell}(1-q^{x+1-\ell})^{a_- +1-\ell}
\nonumber\\
&\times
\prod_{1\leq\ell\leq b_-} (1-q^{x+1/2+\ell})^{b_- -\ell}(1-q^{x+3/2-\ell})^{b_- +1-\ell} .
\label{tauN}
\end{align}
(The proof of this factorization formula will be subsequently dealt with in
Section \ref{sec5} below.)
Clearly, the construction of the Baker-Akhiezer function $\psi_N(x,\hat{z})$ is invariant
with respect to permutations of the spectral data
$(\hat{z}_1,\hat{\nu}_1),\ldots ,(\hat{z}_N,\hat{\nu}_N)$,
and the specialization of parameters
in \eqref{parity-sym} moreover respects this symmetry.
We will now exploit this permutation symmetry to
fix $\hat{z}_N$ from the possible values in \eqref{AWdata} as
\begin{equation}\label{zn}
\hat{z}_N=q^{(n^+_+ +n^+_-)/2}=q^{(a_+ +b_+ +a_- +b_-)/2}.
\end{equation}
The passage from $\tau_N(x)$ to $\tau_{N-1}(x)$ then amounts to eliminating
the value $n^+_+ +n^+_-$ from $I^+$, or equivalently, to a shift of parameters of the
form
$(n^+_+,n^+_-;n^-_+,n^-_-)\longrightarrow
(n^+_+ -1+\delta_{n^+_+,0},n^+_--1+\delta_{n^+_-,0};n^-_+,n^-_-)$
(where $\delta_{j,k}$ refers to the Kronecker delta symbol), which gives rise to
a shift of the parameters of the form $(a_+,b_+;a_-,b_-)\longrightarrow
(b_+ -1+\delta_{b_+,0},a_+;b_- -1+\delta_{b_-,0}, a_-)$ in Theorem \ref{awc:thm}.
The upshot is that for the choice $\hat{z}_N$ \eqref{zn} the tau function
$\tau_{N-1}(x)$
factorizes as
\begin{align}
 \tau_{N-1} (x)
 &=
 \prod_{1\leq\ell\leq b_+ -1} (1+q^{x+\ell})^{b_+ -\ell}(1+q^{x+1-\ell})^{b_+ -\ell}
 \nonumber \\
&\times
\prod_{1\leq\ell\leq a_+} (1+q^{x+1/2+\ell})^{a_+ -\ell}(1+q^{x+3/2-\ell})^{a_+ +1-\ell}
\nonumber \\
&\times
\prod_{1\leq\ell\leq b_- -1} (1-q^{x+\ell})^{b_- -\ell}(1-q^{x+1-\ell})^{b_- -\ell}
\nonumber\\
&\times
\prod_{1\leq\ell\leq a_-} (1-q^{x+1/2+\ell})^{a_- -\ell}(1-q^{x+3/2-\ell})^{a_- +1-\ell} .
\label{tauN-1}
\end{align}
We thus have that
\begin{subequations}
\begin{align}
\lefteqn{A_N(x)B_N(x+1) = \frac{\tau_N(x+1) \tau_N(x-1)}{\tau_N^2(x)} } & \nonumber \\
& =
\frac{(1+q^{x+ 1+a_+ })(1+q^{x  +1/2 + b_+}) (1-q^{x +1+ a_-})(1-q^{x +1/2 +b_-})}
     {(1-q^{2x+1})(1-q^{2x+2})} \nonumber  \\
&\quad \times
\frac{(1+q^{x-a_+})(1+q^{x+1/2-b_+}) (1-q^{x-a_-})(1-q^{x+ 1/2-b_-})}
     {(1-q^{2x})(1-q^{2x+1})}
\end{align}
and that
\begin{align}
A_N(x) + B_N(x) &=
q^{(a_+ +a_- +b_+ +b_-)/2} \times \\
& \frac{(1+q^{x-a_+})(1+q^{x+1/2-b_+}) (1-q^{x-a_-})(1-q^{x+ 1/2-b_-})}
     {(1-q^{2x})(1-q^{2x+1})} + \nonumber  \\
&
q^{-(a_+ +a_- +b_+ +b_-)/2} \times \nonumber \\
& \frac{(1+q^{x+ a_+})(1+q^{x-1/2+ b_+}) (1-q^{x+ a_-})(1-q^{x-1/2+ b_-})}
     {(1-q^{2x})(1-q^{2x-1})} .\nonumber
\end{align}
\end{subequations}
With the aid of the latter expressions, it is immediately seen that
$A_N(x)B_N(x+1)=A(q^x)A(q^{-x-1})$, and
it is also not difficult to check (e.g. by comparing the
residues at the poles and the asymptotics for $x\to+\infty$ on both sides) that
$A_N(x)+B_N(x)-\hat{z}_N-\hat{z}_N^{-1}=A(q^x)+A(q^{-x})-\hat{a}-\hat{a}^{-1}$
(where the parameters are assumed to be of the form in \eqref{int-par}, so
$\hat{a}=q^{1+(a_+ +b_+ +a_- +b_-)/2}=q\hat{z}_N$). We thus reproduce the coefficients
of the Askey-Wilson difference equation, whence the claim follows.

\section{Proof of the Factorization of the Askey-Wilson-Cauchy Determinant}\label{sec5}
In this section we prove the product formula for the Askey-Wilson-Cauchy determinant
$\tau(x|n^+_+,n^+_-;n^-_+,n^-_-):=\det[\mathbf{1}+\mathbf{B}(x)\mathbf{C}]$
given in Theorem \ref{awc:thm}. The argument is self-contained and elementary (albeit technical): apart from the
definition of $\mathbf{B}(x)$ and $\mathbf{C}$, no property of the Askey-Wilson
functions or of the Toda chain enters. Specifically, the proof implements the
\emph{`identification of factors' method} for determinant evaluations from \cite[Section 2.4]{kra:advanced}
by means of the following steps.
\begin{enumerate}
\item[(a)] The product formula of Theorem \ref{awc:thm} is reformulated in terms of
elementary building blocks (`tent polynomials') as Theorem \ref{tent:thm}.
\item[(b)] The expansion \eqref{AWtauexp} of the determinant is rewritten as a signed sum over
subsets in which all absolute values have been eliminated (Proposition \ref{signed:prp}).
\item[(c)] By means of an alternant identity (Lemma \ref{alt:lem}), the signed sum is identified
with the ratio $D(v)/\Delta(y)$ of the determinant of a matrix with polynomial entries in
$v=-q^{x+1/2}$ and a Vandermonde determinant fixing the normalization of the polynomial (Proposition \ref{alt:prp}). This exhibits the
Askey-Wilson-Cauchy determinant as a polynomial in $v$ with constant term $1$ and of known degree
(Proposition \ref{deg:prp}).
\item[(d)] \emph{Identification of factors}: at each candidate zero $v=\sigma p^m$ (with $p=q^{1/2}$,
$\sigma\in\{1,-1\}$ and $m\in\mathbb{Z}$), a reflection symmetry of the columns of the polynomial
matrix yields a lower bound for its corank (Lemma \ref{refl:lem}), and hence for the order of
vanishing of $D(v)$ (Lemma \ref{corank:lem} and Proposition \ref{order:prp}). This bound is computed in
closed form (Lemma \ref{S:lem} and Proposition \ref{W:prp}) and turns out to coincide with the
multiplicity of the corresponding factor on the right-hand side of Theorem \ref{tent:thm}
(Proposition \ref{zeros:prp}).
\item[(e)] \emph{Degree count and multiplicative constant}: the multiplicities of the identified factors add
up to the degree of the determinant (Lemma \ref{count:lem}), so the product of these factors divides the
determinant and has the same degree; a comparison of the constant terms then completes the proof
of Theorem \ref{tent:thm}.
\end{enumerate}
Throughout this section we employ the abbreviations
\begin{equation}\label{put}
p:=q^{1/2},\qquad u:=q^{x+1/2},\qquad v:=-u,
\end{equation}
so $0<p<1$ and $q^{j(x+1/2)}=u^j$. Here the variable $u$ is adapted to the product formula
in \eqref{tau-prod}, whereas $v$ is adapted to the alternant, as the sign configuration
$\epsilon^\pm_j$ turns the weights of the determinantal expansion into powers of $-u$ (cf. Proposition \ref{signed:prp} below).
For definiteness we take $x$ real, so $u>0$;
since both sides of the product formula are polynomials in $u$ (see below),
the formula then extends to all complex $x$.

\subsection{Tent polynomials}
For an integer $M\geq 0$ and an indeterminate $w$ we define the {\em tent polynomial}\footnote{The terminology
refers to the tent-shaped (symmetric, unimodal, and piecewise linear) profile of the multiplicities:
plotted against the exponent $m$, these increase with slope $\frac{1}{2}$ up to the apex and decrease
again, the base of the tent being of width $2M$.}
\begin{subequations}
\begin{equation}\label{tent}
\begin{aligned}
T(M;w)&:=\prod_{m\in\mathbb{Z}}\bigl(1+w\,p^{m}\bigr)^{\mathrm{t}_M(m)},
\\
\mathrm{t}_M(m)&:=\begin{cases}\max\bigl(0,\tfrac{M-|m|}{2}\bigr)&\text{if}\ m\equiv M \ (\mathrm{mod}\ 2),\\[1ex]
0&\text{if}\ m\not\equiv M \ (\mathrm{mod}\ 2) . \end{cases}
\end{aligned}
\end{equation}
In other words, the exponents $m$ occurring in $T(M;w)$ have the
parity of $M$ and satisfy $|m|<M$; the multiplicity $\mathrm{t}_M(m)$ is
$\lfloor M/2\rfloor$ at $m=0$ (if $M$ is even) or at $m=\pm1$ (if $M$ is odd), and it decreases by one at each step of size $2$ in $|m|$.
In particular $T(0;w)=T(1;w)=1$, $T(2;w)=1+w$, and $\mathrm{t}_M(-m)=\mathrm{t}_M(m)$.
Since $q^{x+\ell}=up^{2\ell-1}$, $q^{x+1-\ell}=up^{1-2\ell}$,
$q^{x+1/2+\ell}=up^{2\ell}$ and $q^{x+3/2-\ell}=up^{2-2\ell}$, a
comparison of the exponents and their multiplicities reveals that,
for $\sigma\in\{1,-1\}$, the four products on the right-hand side of \eqref{tau-prod} are given by
\begin{align}
\prod_{1\leq\ell\leq a} (1+\sigma q^{x+\ell})^{a +1-\ell}(1+\sigma q^{x+1-\ell})^{a +1-\ell}&=T(2a+1;\sigma u),\nonumber
\\
\prod_{1\leq\ell\leq b} (1+\sigma q^{x+1/2+\ell})^{b -\ell}(1+\sigma q^{x+3/2-\ell})^{b +1-\ell}&=T(2b;\sigma u). \label{tent-prod}
\end{align}
\end{subequations}
Moreover, for the parameters $a_\pm,b_\pm$ from \eqref{par-rel} one has that
$\{2a_\pm+1,2b_\pm\}=\{n^+_\pm+n^-_\pm+1,|n^+_\pm-n^-_\pm|\}$ (as unordered pairs), in both parity cases.
Hence Theorem \ref{awc:thm} is equivalent to the following statement.

\begin{theorem}[Tent Formula]\label{tent:thm}
For all nonnegative integers $n^+_+,n^+_-,n^-_+,n^-_-$,
\begin{align}
\lefteqn{\tau(x|n^+_+,n^+_-;n^-_+,n^-_-)=} & \label{tent-thm}
\\
&T(n^+_++n^-_++1;u)\,T(|n^+_+-n^-_+|;u)\,
T(n^+_-+n^-_-+1;-u)\,T(|n^+_--n^-_-|;-u).\nonumber
\end{align}
\end{theorem}

Notice the pairing of the parameters in \eqref{tent-thm}:
the tent polynomials in $u$ involve $n^+_+$ and $n^-_+$, whereas those in $-u$ involve
$n^+_-$ and $n^-_-$. The remainder of this section is devoted to the proof of Theorem \ref{tent:thm}.

\subsection{Combinatorics of the polynomial expansion for $\tau$}\label{sec5:not}
For $\delta\in\{+,-\}$, let us abbreviate
\begin{equation*}
d_\delta:=|n^\delta_+-n^\delta_-|,\quad \mu_\delta:=\min(n^\delta_+,n^\delta_-),\quad N_\delta:=|I^\delta| =\max (n^\delta_+,n^\delta_-),
\end{equation*}
and put
\begin{equation}\label{eps-eta}
N:=N_++N_-,\quad \epsilon^\delta_j=(-1)^{\eta^\delta_j+\frac{1}{2}(1+\theta_\delta)j}\quad\text{with}\
 \theta_\delta:= \begin{cases} 1&\text{if}\  n^\delta_+\geq n^\delta_-, \\ -1&\text{if}\  n^\delta_+<n^\delta_- ,\end{cases}
\end{equation}
where (recall) $I^\delta=\{1,\ldots ,d_\delta\}\cup\{d_\delta+2,d_\delta+4,\ldots ,n^\delta_++n^\delta_-\}$ and $\eta^\delta_j$ denotes
the position of $j$ in $I^\delta$.

The expansion of the determinant \eqref{AWtauexp} will now be rewritten in the polynomial form (cf. \eqref{bafs2})
\begin{equation}\label{sum0}
\tau=\sum_{\mathcal{J}\subseteq \mathcal{I}}\ \prod_{i\in \mathcal{J}}\epsilon_i\,u^{j_i}\prod_{i\in \mathcal{J},\,i'\not\in \mathcal{J}}|R_{ii'}| ,
\end{equation}
where we have suppressed the dependence on $x$ and on $n^\pm_\pm$.
Here the sum runs over the subsets of the disjoint union $\mathcal{I}:=I^+\sqcup I^-$, which labels the rows and columns of
$\mathbf{B}(x)$ and $\mathbf{C}$. An element $i\in \mathcal{I}$ carries a sign $\delta(i)\in\{+,-\}$ and a label $j_i\in I^{\delta(i)}$, and we write
$\epsilon_i:=\epsilon^{\delta(i)}_{j_i}$ and $\eta_i:=\eta^{\delta(i)}_{j_i}$. Furthermore, using $s_+:=1$ and $s_-:=-1$:
\begin{equation}\label{yR}
 R_{ii'}:=\frac{1-y_iy_{i'}}{y_i-y_{i'}}\quad (i\neq i')\qquad \text{with}\quad  y_i:=s_{\delta(i)}\,p^{\,j_i} .
\end{equation}
Clearly the numbers $y_i$, $i\in \mathcal{I}$, are nonzero and pairwise distinct, and $R_{i'i}=-R_{ii'}$.
Notice that in this combinatorial reformulation $\mathbf{C}=[(1-y_iy_{i'})^{-1}]_{i,i'\in \mathcal{I}}$, and factors of the form
$|(1-q^{(j_i+j_{i'})/2})/(q^{j_i/2}-q^{j_{i'}/2})|$ and $(1+q^{(j_i+j_{i'})/2})/(q^{j_i/2}+q^{j_{i'}/2})$
in $\mathbf{B}(x)$ are equal to $|R_{ii'}|$ (for $i,i'$ of equal and of opposite sign, respectively), which
confirms \eqref{sum0}.

\subsection{Removing the absolute values}
\begin{lemma}[Sign Bookkeeping]\label{signs:lem}
Let $i, i'$ be two distinct elements of $\mathcal{I}$ and let $\delta\in \{ +,-\}$.

 (a) If $\delta(i)=\delta(i')=\delta$ and $j_i<j_{i'}$, then $R_{ii'}$ has the sign $s_\delta$.
 
(b) If $\delta(i)=\delta$ and $\delta(i')\neq \delta$, then $R_{ii'}$ has the sign $s_\delta$.
\end{lemma}
\begin{proof}
 For $i,i'\in I^+$ with $j_i<j_{i'}$ one has $R_{ii'}=(1-p^{j_i+j_{i'}})/(p^{j_i}-p^{j_{i'}})>0$ (as $0<p<1$), while replacing $(y_i,y_{i'})$ by $(-y_i,-y_{i'})$ flips the sign of $R_{ii'}$, so (a) follows.
For $i\in I^+$ and $i'\in I^-$ one has $R_{ii'}=(1+p^{j_i +j_{i'}})/(p^{j_i}+p^{{j_{i'}}})>0$, so (b) follows by the antisymmetry $R_{i'i}=-R_{ii'}$.
\end{proof}

\begin{proposition}[Sign Pattern]\label{signed:prp}
One has
\begin{subequations}
\begin{equation}\label{sum1}
\tau=\sum_{\mathcal{J}\subseteq \mathcal{I}}(-1)^{\binom{|\mathcal{J}|+1}{2}}\prod_{i\in \mathcal{J}}\alpha_i\prod_{i\in \mathcal{J},\,i'\not\in \mathcal{J}}R_{ii'} ,
\end{equation}
with
\begin{equation}\label{alpha}
\alpha_i:=\begin{cases}(\theta_+v)^{j_i}&\text{if}\ i\in I^+,\\[.5ex] (-1)^{N-1}(\theta_-v)^{j_i}&\text{if}\ i\in I^- .\end{cases}
\end{equation}
\end{subequations}
\end{proposition}
\begin{proof}
Fix $\mathcal{J}\subseteq \mathcal{I}$ and put $\mathcal{J}_\delta:=\mathcal{J}\cap I^\delta$ and $r_\delta:=|\mathcal{J}_\delta|$ for $\delta\in\{+,-\}$. For $i\in \mathcal{J}_\delta$, i.e. $\delta=\delta(i)$, let $\eta_i(\mathcal{J})\in\{1,\ldots ,r_\delta\}$ denote the position of $i$ in $\mathcal{J}_\delta$ (with respect to increasing labels). Since $\eta_i$ is the position of $j_i$ in $I^\delta$, we have
\begin{subequations}
\begin{align}
\label{count-lt} \#\{i'\in I^\delta\setminus \mathcal{J}_\delta:\ j_{i'}<j_i\}&=\eta_i-\eta_i(\mathcal{J}),\\
\label{count-gt} \#\{i'\in I^\delta\setminus \mathcal{J}_\delta:\ j_{i'}>j_i\}&=(N_\delta-\eta_i)-(r_\delta-\eta_i(\mathcal{J})).
\end{align}
\end{subequations}
Notice that, as $i$ runs through $\mathcal{J}_\delta$, the position $\eta_i(\mathcal{J})$
runs through $1,\ldots ,r_\delta$, whence $\sum_{i\in \mathcal{J}_\delta}\eta_i(\mathcal{J})=\binom{r_\delta+1}{2}$.
We are now in a position to count the negative factors among the $R_{ii'}$ with $i\in \mathcal{J}$, $i'\not\in \mathcal{J}$ by means of Lemma \ref{signs:lem}:
\begin{itemize}
\item[$-$] for $i\in \mathcal{J}_+$, $i'\in I^+\setminus \mathcal{J}_+$ the factor is negative iff $j_{i'}<j_i$, so \eqref{count-lt} yields
$\sum_{i\in \mathcal{J}_+}\eta_i-\binom{r_++1}{2}$ negative factors;
\item[$-$] for $i\in \mathcal{J}_-$, $i'\in I^-\setminus \mathcal{J}_-$ the factor is negative iff $j_{i'}>j_i$, so \eqref{count-gt} yields
$r_-(N_--r_-)-\sum_{i\in \mathcal{J}_-}\eta_i+\binom{r_-+1}{2}$ negative factors;
\item[$-$] for $i\in \mathcal{J}_+$, $i'\in I^-\setminus \mathcal{J}_-$ there are no negative factors;
\item[$-$] for $i\in \mathcal{J}_-$, $i'\in I^+\setminus \mathcal{J}_+$ all $r_-(N_+-r_+)$ factors are negative.
\end{itemize}
By summing all contributions, it is seen that
modulo $2$ the total number of negative factors is congruent to
\begin{equation*}
\sum_{i\in \mathcal{J}}\eta_i+\binom{r_++1}{2}+\binom{r_-+1}{2}+r_+r_-+r_-N+r_-^2\equiv
\sum_{i\in \mathcal{J}}\eta_i+\binom{|\mathcal{J}|+1}{2}+r_-(N-1),
\end{equation*}
where we used that $\binom{r_++1}{2}+\binom{r_-+1}{2}+r_+r_-=\binom{r_++r_-+1}{2}$ and $r_-^2\equiv r_-\pmod 2$ (as $r_-(r_--1)$ is even). The upshot is that
\begin{equation*}
\prod_{i\in \mathcal{J}}\epsilon_iu^{j_i}\prod_{i\in \mathcal{J},\,i'\not\in \mathcal{J}}|R_{ii'}|=
(-1)^{\binom{|\mathcal{J}|+1}{2}}\prod_{i\in \mathcal{J}}\epsilon_i(-1)^{\eta_i}u^{j_i}\prod_{i\in \mathcal{J}_-}(-1)^{N-1}\prod_{i\in \mathcal{J},\,i'\not\in \mathcal{J}}R_{ii'},
\end{equation*}
which yields \eqref{sum1}, \eqref{alpha} upon observing that, by \eqref{eps-eta},
$\epsilon_i(-1)^{\eta_i}u^{j_i}=(-u)^{j_i}=v^{j_i}$ if $\theta_{\delta(i)}=1$ and
$\epsilon_i(-1)^{\eta_i}u^{j_i}=u^{j_i}=(-v)^{j_i}$ if $\theta_{\delta(i)}=-1$, i.e. $\epsilon_i(-1)^{\eta_i}u^{j_i}=(\theta_{\delta(i)}v)^{j_i}$ in both cases.
\end{proof}

\subsection{Alternant representation for $\tau$}\label{sec5:alt}
So as to write down determinants with rows labeled by $\mathcal{I}$,
we fix an arbitrary total ordering of $\mathcal{I}$. Since such determinants change by a common sign upon reordering, their ratios do not depend on the choice of ordering.

\begin{lemma}[Alternant]\label{alt:lem}
For arbitrary parameters $\alpha_1,\ldots ,\alpha_N$ and distinct nonzero (possibly complex) variables $z_1,\ldots ,z_N$, one has
\begin{equation}\label{alt}
\det\bigl[z_r^c-\alpha_rz_r^{N-1-c}\bigr]_{\begin{subarray}{c}1\leq r\leq N\\ 0\leq c< N\end{subarray}}
=\Delta(z)\sum_{\mathcal{J}\subseteq\{1,\ldots ,N\}}(-1)^{\binom{|\mathcal{J}|+1}{2}}\prod_{r\in \mathcal{J}}\alpha_r\prod_{\begin{subarray}{c}r\in \mathcal{J}\\ r'\not\in \mathcal{J}\end{subarray}}\frac{1-z_rz_{r'}}{z_r-z_{r'}} ,
\end{equation}
where
$\Delta(z):=\det[z_r^c]_{1\leq r\leq N,\,0\leq c< N}=\prod_{1\leq r<r'\leq N}(z_{r'}-z_r)$ denotes the Vandermonde determinant.
\end{lemma}
\begin{proof}
Row $r$ of the matrix on the left is of the form $\rho(z_r)-\alpha_rz_r^{N-1}\rho(z_r^{-1})$ with $\rho(w):=(w^c)_{0\leq c< N}$. Our determinant thus expands by multilinearity in the rows
in terms of Vandermonde determinants as follows:
\begin{subequations}
\begin{equation}\label{A-exp}
\det\bigl[z_r^c-\alpha_rz_r^{N-1-c}\bigr]_{\begin{subarray}{c}1\leq r\leq N\\ 0\leq c< N\end{subarray}}=
\sum_{\mathcal{J}\subseteq\{1,\ldots ,N\}}\prod_{r\in \mathcal{J}}(-\alpha_rz_r^{N-1})\,\Delta(z^{\mathcal{J}}),
\end{equation}
where $z^{\mathcal{J}}$ denotes $z$ with $z_r$ replaced by $z_r^{-1}$ for $r\in \mathcal{J}$. Upon
comparing $\Delta(z^{\mathcal{J}})=\prod_{r<r'}(z^{\mathcal{J}}_{r'}-z^{\mathcal{J}}_r)$ with $\Delta(z)$ factor by factor, one finds:
\begin{equation}\label{A-term}
\Delta(z^{\mathcal{J}})=\Delta(z)(-1)^{\binom{|\mathcal{J}|}{2}}\prod_{r\in \mathcal{J}}z_r^{-(N-1)}\prod_{r\in \mathcal{J},\,r'\not\in \mathcal{J}}\frac{1-z_rz_{r'}}{z_r-z_{r'}} .
\end{equation}
\end{subequations}
Indeed, for $r<r'$ both in $\mathcal{J}$, $z_{r'}^{-1}-z_r^{-1}=-(z_rz_{r'})^{-1}(z_{r'}-z_r)$; for a pair $\{ r,r'\}$ with $r\in \mathcal{J}$, $r'\not\in \mathcal{J}$ (in either order),
$z_{r'}-z_r^{-1}=\frac{1-z_rz_{r'}}{z_r(z_r-z_{r'})}(z_{r'}-z_r)$; and pairs outside $\mathcal{J}$ are unchanged.
Moreover, a given $r\in \mathcal{J}$ lies in $|\mathcal{J}|-1$ pairs of the first kind and in $N-|\mathcal{J}|$ pairs of the second kind, so \eqref{A-term} follows.
Finally, upon combining \eqref{A-exp} and \eqref{A-term} one arrives at the asserted expansion
by virtue of the identity $(-1)^{|\mathcal{J}|}(-1)^{\binom{|\mathcal{J}|}{2}}=(-1)^{\binom{|\mathcal{J}|+1}{2}}$.
\end{proof}

\begin{proposition}[Alternant Representation]\label{alt:prp}
Let $\mathbf{A}(v):=[A_{i,c}(v)]_{i\in \mathcal{I},\,0\leq c< N}$ be the $N\times N$ matrix with entries
\begin{subequations}
\begin{equation}\label{rows}
A_{i,c}(v)=s_\delta^{\,c}\Bigl[p^{jc}-(\theta_\delta v)^{j}p^{j(N-1-c)}\Bigr],\qquad \delta=\delta(i),\ j=j_i .
\end{equation}
Then
\begin{equation}
\tau=\frac{D(v)}{\Delta(y)} ,
\end{equation}
where $D(v):=\det\mathbf{A}(v)$ and $\Delta(y):=\det[y_i^c]_{i\in \mathcal{I},\,0\leq c< N}$ ($\neq 0$).
\end{subequations}
\end{proposition}
\begin{proof}
We apply Lemma \ref{alt:lem} with $z=y$ and with the weights $\alpha_i$ of \eqref{alpha}. Since $(1-y_iy_{i'})/(y_i-y_{i'})=R_{ii'}$, the right-hand side of \eqref{alt} now equals $\Delta(y)$ times the right-hand side of \eqref{sum1}. The rows of the matrix on the left-hand side of \eqref{alt} are $y_i^c-\alpha_iy_i^{N-1-c}=s_\delta^{\,c}p^{jc}-\alpha_is_\delta^{\,N-1-c}p^{j(N-1-c)}$. For $\delta=+$ this is precisely \eqref{rows}. For $\delta=-$ we have $\alpha_is_-^{N-1-c}=(-1)^{N-1}(\theta_-v)^j(-1)^{N-1-c}=(-1)^c(\theta_-v)^j$, which again gives \eqref{rows}. (The sign $(-1)^{N-1}$ in $\alpha_i$ was included precisely to compensate the sign of $(-p^j)^{N-1-c}$.)
\end{proof}

\subsection{Degree and constant term of $\tau$}
\begin{proposition}[Degree and Constant Term]\label{deg:prp}
$\tau$ is a polynomial in $v$ with constant term $1$ and of degree
\begin{equation}\label{deg}
\deg_v\tau=\sum_{i\in \mathcal{I}}j_i=\sum_{\delta\in\{+,-\}}\left[\binom{n^\delta_++1}{2}+\binom{n^\delta_-+1}{2}\right] .
\end{equation}
\end{proposition}
\begin{proof}
In \eqref{sum1} the factors $R_{ii'}$ do not depend on $v$, so the term corresponding to $\mathcal{J}$ is a constant multiple of $v^{\sum_{i\in \mathcal{J}}j_i}$. The term with $\mathcal{J}=\emptyset$ equals $1$, and the maximal exponent $\sum_{i\in \mathcal{I}}j_i$ is attained only for $\mathcal{J}=\mathcal{I}$, with coefficient $\pm 1$.
Finally, from $I^\delta=\{1,\ldots ,d_\delta\}\cup\{d_\delta+2,\ldots ,d_\delta+2\mu_\delta\}$ one computes
$\sum_{j\in I^\delta}j=\binom{d_\delta+1}{2}+\mu_\delta d_\delta+\mu_\delta(\mu_\delta+1)=\binom{n^\delta_++1}{2}+\binom{n^\delta_-+1}{2}$
(as $\{n^\delta_+,n^\delta_-\}=\{\mu_\delta,\mu_\delta+d_\delta\}$).
\end{proof}

Note that,
since $\Delta(y)$ does not depend on $v$, the degree of the alternant $D(v)=\Delta(y)\tau$ is equal to the degree of $\tau$ (while its constant term is equal to $\Delta(y)$).

\subsection{Coranks, reflection symmetries, and the multiplicity of zeros}
The following elementary lemma is a direct consequence of the Smith normal form of a polynomial matrix,
cf. e.g. \cite[Chapter VI]{gan:theory}. It permits the passage from a kernel of dimension $k$ to a zero of order $k$ in the determinant, and as such it constitutes a crucial
tool for detecting factors of higher multiplicity in determinant evaluations, cf. \cite[Section 2.4]{kra:advanced} and \cite[Section 2]{kra:andrews-burge}.

\begin{lemma}[Corank Bound]\label{corank:lem}
Let $\mathbf{A}(v)$ be an $N\times N$ matrix whose entries are polynomials in $v$, and let $v_0\in\mathbb{C}$. If
$\operatorname{corank}\mathbf{A}(v_0)\geq k\in\mathbb{N}_0$, then $(v-v_0)^k$ divides $\det\mathbf{A}(v)$.
\end{lemma}
\begin{proof}
Pick an invertible constant matrix $\mathbf{Q}$ whose last $k$ columns lie in the kernel of $\mathbf{A}(v_0)$. The last $k$ columns of $\mathbf{A}(v)\mathbf{Q}$ vanish at $v=v_0$, so each of their entries is divisible by $v-v_0$; pulling out these factors gives $(v-v_0)^k\mid\det(\mathbf{A}(v)\mathbf{Q})=\det\mathbf{A}(v)\det\mathbf{Q}$.
\end{proof}

For square matrices $\mathbf{A}$
that intertwine a permutation representation of the reflection group $\mathbb{Z}_2$ on a subspace of the column space with a sign representation on the row space, the following lemma provides a lower bound on the corank of $\mathbf{A}$.

\begin{lemma}[Coranks and Reflection Symmetries]\label{refl:lem}
Let $\mathbf{A}=[A_{r,c}]$ be an $N\times N$ matrix with rows labeled by $r\in\{1,\ldots ,N\}$ and columns labeled by $c\in\{0,\ldots ,N-1\}$. Let $J\subseteq\{0,\ldots ,N-1\}$ be stable under an involution $c\mapsto c'$ with $P$ two-element orbits and $F$ fixed points, and assume that for every row $r$ there is a sign $\lambda_r\in\{1,-1\}$ such that $A_{r,c'}=\lambda_rA_{r,c}$ for all $c\in J$. Let $R_\pm:=\#\{r:\lambda_r=\pm1\}$. Then
\begin{equation*}
\operatorname{corank}\mathbf{A}\geq\max\bigl(P+F-R_+,\ P-R_-\bigr).
\end{equation*}
\end{lemma}
\begin{proof}
Let $\mathbf{e}_c$ denote the standard basis vectors. The vectors $\mathbf{e}_c+\mathbf{e}_{c'}$ (one for each two-element orbit $\{c,c'\}$) together with the vectors $\mathbf{e}_c$ (one for each fixed point $c$) involve pairwise disjoint indices and thus span a subspace $V_+$ of dimension $P+F$. For an orbit $\{c,c'\}$, the $r$-th entry of $\mathbf{A}(\mathbf{e}_c+\mathbf{e}_{c'})$ is $(1+\lambda_r)A_{r,c}$. For a fixed point $c$, on the other hand, one has $A_{r,c}=\lambda_rA_{r,c}$, which implies that $A_{r,c}=0$ whenever $\lambda_r=-1$. Hence $\mathbf{A}$ maps $V_+$ into the coordinate subspace spanned by the rows with $\lambda_r=1$, which has dimension $R_+$, and therefore $\dim(\ker\mathbf{A}\cap V_+)\geq P+F-R_+$. Similarly, the $P$ vectors $\mathbf{e}_c-\mathbf{e}_{c'}$ span a subspace $V_-$ that is mapped by $\mathbf{A}$ into the coordinate subspace of the rows with $\lambda_r=-1$ (the $r$-th entry of $\mathbf{A}(\mathbf{e}_c-\mathbf{e}_{c'})$ being $(1-\lambda_r)A_{r,c}$), whence $\dim(\ker\mathbf{A}\cap V_-)\geq P-R_-$.
\end{proof}

\subsection{The corank bound at $v=\sigma p^m$}
The asserted factorization of $\tau$ in Theorem \ref{tent:thm} anticipates that the roots of the alternant $D(v)$ in Proposition \ref{alt:prp} should only arise at certain values of $v$ of the form
 $v_0:=\sigma p^m$, where $\sigma\in\{1,-1\}$ and $m\in\mathbb{Z}$ with $|m|<N$.
With the aid of Lemmas \ref{corank:lem} and \ref{refl:lem}, we will now deduce a lower bound for the order of
vanishing of $D(v)$ at the points in question (not every point of this form need be an actual zero).
To this end let us define
\begin{equation}\label{J}
\begin{split}
\ell &:=m+N-1, \qquad c':=\ell-c, \\
J &:=\Bigl\{c\in\{0,\ldots ,N-1\}: \  c'\in\{0,\ldots ,N-1\}\Bigr\} .
\end{split}
\end{equation}
Explicitly, $J=\{\max(0,m),\ldots ,\min(N-1,m+N-1)\}$, so that $h:=|J|=N-|m|$. Notice that the involution $c\mapsto c'$ of $J$, which encodes the reflection of the column indices about $\ell/2$, has $P=\lfloor h/2\rfloor$ two-element orbits and $F$ fixed points, where $F:=(h\bmod 2)\in \{0,1\}$.

For a row $i\in\mathcal{I}$ with sign $\delta$ and label $j$ we abbreviate $\kappa_i:=(\theta_\delta\sigma)^j$. Since
\begin{equation*}
(\theta_\delta v_0)^jp^{j(N-1-c)}=\kappa_ip^{j(m+N-1-c)}=\kappa_ip^{jc'} ,
\end{equation*}
it is seen from \eqref{rows} that in this notation
$A_{i,c}(v_0)=s_\delta^{\,c}[p^{jc}-\kappa_ip^{jc'}]$ for all $c$. In particular, for $c\in J$ the reflection $c\mapsto c'$ changes the entries of the $i$-th row at most by a sign:
\begin{equation}\label{lambda}
\begin{split}
A_{i,c'}(v_0)=s_\delta^{\,c'}\bigl[p^{jc'}-\kappa_ip^{jc}\bigr]=\lambda_iA_{i,c}(v_0),\\
\text{with}\ \lambda_i:=-\kappa_is_\delta^{\,\ell}=-(\theta_\delta\sigma)^{j}s_\delta^{\,\ell}\in\{1,-1\}
\end{split}
\end{equation}
(where we used that $\kappa_i^2=1$ and $s_\delta^{\,c+c'}=s_\delta^{\,\ell}$).
Notice that the row sign $\lambda_i$ depends on the row $i$ only through its sign $\delta(i)$ and the parity of its label $j_i$, which is what will allow us to evaluate the sum of the row signs in closed form in Subsection \ref{sec5:M} below.

\begin{proposition}[Vanishing Order Bound]\label{order:prp}
Let $\sigma\in\{1,-1\}$ and $m\in\mathbb{Z}$ with $|m|<N$, let $v_0:=\sigma p^m$ and $F:=(N-|m|)\bmod 2$, and let $\Lambda:=\sum_{i\in \mathcal{I}}\lambda_i$ with $\lambda_i$ as in \eqref{lambda}. Then
\begin{subequations}
\begin{equation}\label{bound}
\operatorname{ord}_{v=v_0}D(v)\geq\max\Bigl(0,\tfrac12\bigl(M-|m|\bigr)\Bigr),\quad\text{with}\  M:=|F-\Lambda| ,
\end{equation}
where moreover
\begin{equation}\label{Wpar}
M\equiv m\pmod 2
\end{equation}
\end{subequations}
(so the lower bound in \eqref{bound} is an integer).
\end{proposition}
\begin{proof}
By Lemma \ref{refl:lem}, applied to $\mathbf{A}(v_0)$ with the rows labeled by $\mathcal{I}$ in the ordering fixed in Subsection \ref{sec5:alt}, we have that
\begin{equation*}
\operatorname{corank}\mathbf{A}(v_0)\geq\max(\lceil h/2\rceil-R_+,\lfloor h/2\rfloor-R_-) ,
\end{equation*}
where the two alternatives stem from the symmetric combinations of the paired columns together with the unpaired column when $h$ is odd, and from the antisymmetric combinations, respectively (which of the two is sharper depends on the sign of $F-\Lambda$).
Since $\Lambda=R_+-R_-$ and $R_++R_-=N$, one has $R_\pm=\frac12(N\pm\Lambda)$; using $\lceil h/2\rceil=\frac12(h+F)$, $\lfloor h/2\rfloor=\frac12(h-F)$ and $h-N=-|m|$, we thus get
$\lceil h/2\rceil-R_+=\frac12(-|m|+(F-\Lambda))$ and $\lfloor h/2\rfloor-R_-=\frac12(-|m|-(F-\Lambda))$, whence
$\operatorname{corank}\mathbf{A}(v_0)\geq \max\bigl(0,\frac12(M-|m|)\bigr)$.
Since $\Lambda\equiv R_++R_-=N$ and $F\equiv h\equiv N-m \pmod 2$, we have moreover that $M\equiv F-\Lambda\equiv m\pmod 2$, which is \eqref{Wpar}.
In particular the lower bound is a nonnegative integer, so Lemma \ref{corank:lem} yields \eqref{bound}.
\end{proof}

\subsection{Computation of the tent size $M$}\label{sec5:M}
For $\varepsilon\in\{1,-1\}$ and $\delta\in\{+,-\}$ we put $S_\delta(\varepsilon):=\sum_{j\in I^\delta}\varepsilon^j$, and we abbreviate $S_\delta:=S_\delta(\theta_\delta\sigma)$.
Summing $\lambda_i$ \eqref{lambda} over the rows gives $\Lambda=-S_+-(-1)^{\ell}S_-$.
Now $(-1)^\ell=(-1)^{m+N-1}$ is equal to $-1$ if $m\equiv N$ (in which case $h$ is even and $F=0$), and equal to $1$ if $m\not\equiv N$ (in which case $F=1$). With the notation $\bar{n}:=-1-n$ (for integers $n$) we therefore have
\begin{equation}\label{W}
M=|F-\Lambda|=\begin{cases}|S_+-S_-|&\text{if}\ m\equiv N\pmod 2,\\[.5ex]
|1+S_++S_-|=|S_+-\overline{S_-}|&\text{if}\ m\not\equiv N\pmod 2 .\end{cases}
\end{equation}

\begin{lemma}[Tent Size Parameters]\label{S:lem}
For $\sigma\in \{ 1,-1\}$, let $k_\delta:=n^\delta_+$ if $\sigma=1$ and $k_\delta:=n^\delta_-$ if $\sigma=-1$ (not to be confused with the parameters $a_\pm$ of Theorem \ref{awc:thm}). Then $S_\delta(\theta_\delta\sigma)\in\{k_\delta,\overline{k_\delta}\}$.
\end{lemma}
\begin{proof}
If $\theta_\delta\sigma=1$, then either $\sigma=1$ and $n^\delta_+\geq n^\delta_-$, or $\sigma=-1$ and $n^\delta_+<n^\delta_-$; in both cases $k_\delta=\max(n^\delta_+,n^\delta_-)=N_\delta$, and $S_\delta(1)=|I^\delta|=N_\delta=k_\delta$.
If $\theta_\delta\sigma=-1$, then either $\sigma=1$ and $n^\delta_+<n^\delta_-$, or $\sigma=-1$ and $n^\delta_+\geq n^\delta_-$; in both cases $k_\delta=\mu_\delta=\min(n^\delta_+,n^\delta_-)$. The labels $1,\ldots ,d_\delta$ contribute $\sum_{j=1}^{d_\delta}(-1)^j$ to $S_\delta(-1)$, which is $0$ if $d_\delta$ is even and $-1$ if $d_\delta$ is odd. The remaining $\mu_\delta$ labels $d_\delta+2,\ldots ,d_\delta+2\mu_\delta$ all have the parity of $d_\delta$ and contribute $\mu_\delta(-1)^{d_\delta}$. Hence $S_\delta(-1)=\mu_\delta=k_\delta$ if $d_\delta$ is even and $S_\delta(-1)=-1-\mu_\delta=\overline{k_\delta}$ if $d_\delta$ is odd.
\end{proof}

\begin{proposition}[Tent Size]\label{W:prp}
With $k_\pm$ as specified in Lemma \ref{S:lem}, one has
\begin{equation}\label{Wfinal}
M=\begin{cases}k_++k_-+1&\text{if}\ m\equiv k_++k_-+1\pmod 2,\\[.5ex] |k_+-k_-|&\text{if}\ m\equiv k_++k_-\pmod 2 .\end{cases}
\end{equation}
\end{proposition}
\begin{proof}
For integers $a,b\geq 0$ one has $|a-b|=|\bar a-\bar b|$ and $|a-\bar b|=|\bar a-b|=a+b+1$. By \eqref{W} and Lemma \ref{S:lem}, $M$ is of the form $|n_+-n_-|$ with $n_+\in\{k_+,\overline{k_+}\}$ and $n_-\in\{k_-,\overline{k_-}\}$, so $M\in\{|k_+-k_-|,k_++k_-+1\}$. These two numbers have opposite parities, and $M\equiv m$ by \eqref{Wpar}.
\end{proof}

\subsection{Comparison with the tent polynomials, and conclusion of the proof}

\begin{proposition}[Comparison of Vanishing Orders]\label{zeros:prp}
Let $\sigma\in\{1,-1\}$ and $m\in\mathbb{Z}$, and let $k_\pm$ be as specified in Lemma \ref{S:lem}. Then $D(v)$ vanishes at $v=v_0=\sigma p^m$ at least to the order with which the factor $1+\sigma up^{-m}=1-\sigma p^{-m}v$ occurs in
$T(k_++k_-+1;\sigma u)\,T(|k_+-k_-|;\sigma u)$.
\end{proposition}
\begin{proof}
Given an integer $M'\geq 0$, the factor $1+\sigma up^{-m}$ occurs in $T(M';\sigma u)$ with multiplicity $\mathrm{t}_{M'}(-m)=\mathrm{t}_{M'}(m)$, which by \eqref{tent} equals $\max(0,\frac12(M'-|m|))$ if $m\equiv M'$ and $0$ otherwise. Exactly one of the two numbers $M'=k_++k_-+1$ and $M'=|k_+-k_-|$ has the parity of $m$, and by \eqref{Wfinal} this number is the tent size $M$. Hence the multiplicity in question equals $\max(0,\frac12(M-|m|))$. If this multiplicity is positive, then $|m|\leq M-2\leq k_++k_--1\leq N_++N_--1=N-1$ (as $k_\delta\leq N_\delta$), so Proposition \ref{order:prp} applies and yields the claim.
\end{proof}

\begin{lemma}[Degree Count for the Tent Formula]\label{count:lem}
A tent polynomial has the degree $\deg_v T(M;\sigma u)=\lfloor M^2/4\rfloor$. Consequently, for all integers $a,b\geq 0$,
\begin{equation}\label{degtent}
\deg_v T(a+b+1;\sigma u)+\deg_v T(|a-b|;\sigma u)=\binom{a+1}{2}+\binom{b+1}{2} .
\end{equation}
\end{lemma}
\begin{proof}
Counted with multiplicities, the number of factors of $T(M;\sigma u)$ equals the sum of the multiplicities $\mathrm{t}_M(m)$ over $m\in\mathbb{Z}$.
For $M$ odd the exponents are the odd $m$ with $|m|\leq M-2$, with multiplicities $\tfrac{M-1}{2},\ldots ,1$ on either side of $0$, for a total of $\tfrac14 (M^2-1)$. For $M$ even the multiplicity is $\tfrac{M}{2}$ at $m=0$ and $\tfrac{M}{2}-1,\ldots ,1$ on either side, for a total of $\tfrac14 M^2$. In both cases the total is equal to $\lfloor M^2/4\rfloor$.
Since exactly one of the numbers $a+b+1$ and $|a-b|$ is even, exactly one of the two floors on the left-hand side of \eqref{degtent} is exact, whence this left-hand side equals $\tfrac14\bigl((a+b+1)^2+(a-b)^2-1\bigr)=\tfrac12(a^2+a+b^2+b)$ in either case.
\end{proof}

\begin{proof}[Proof of Theorem \ref{tent:thm}]
Now that we have collected all required data, we are in a position to assemble the
proof of the Tent Formula in Theorem \ref{tent:thm} using Krattenthaler's `identification of factors' method.
Specifically, let $\mathcal{T}$ denote the right-hand side of \eqref{tent-thm}, which is a polynomial in $v=-u$ with constant term $1$ whose zeros are located at the points $v=\sigma p^m$, $\sigma\in\{1,-1\}$, $m\in\mathbb{Z}$. These points are pairwise distinct since $0<p<1$. For $\sigma=1$ the pair $(k_+,k_-)$ of Lemma \ref{S:lem} is $(n^+_+,n^-_+)$ and for $\sigma=-1$ it is $(n^+_-,n^-_-)$, so Proposition \ref{zeros:prp} states that $D(v)=\Delta(y)\tau$ vanishes at each zero of $\mathcal{T}$ at least to the order of that zero in $\mathcal{T}$. In other words, $\mathcal{T}$ divides the polynomial $\tau$. Comparing Lemma \ref{count:lem} (applied to $(a,b)=(n^+_+,n^-_+)$ and $(a,b)=(n^+_-,n^-_-)$) with Proposition \ref{deg:prp}, we find that
\begin{equation*}
\deg_v\mathcal{T}=\binom{n^+_++1}{2}+\binom{n^-_++1}{2}+\binom{n^+_-+1}{2}+\binom{n^-_-+1}{2}=\deg_v\tau .
\end{equation*}
Hence $\tau=\gamma\,\mathcal{T}$ for some constant $\gamma$, and a comparison of the constant terms yields $\gamma=1$.
\end{proof}

This completes the proof of Theorem \ref{tent:thm}, and thus of Theorem \ref{awc:thm}.

\begin{remark}
The mechanism behind the factorization is visible in \eqref{W}--\eqref{Wfinal}: the zeros at $v=p^m$ (i.e. the factors $1+up^{-m}$) only involve the sums $S_\delta(\theta_\delta)$, which by Lemma \ref{S:lem} are governed by $n^+_+$ and $n^-_+$, whereas the zeros at $v=-p^m$ (i.e. the factors $1-up^{-m}$) only involve $S_\delta(-\theta_\delta)$, which are governed by $n^+_-$ and $n^-_-$. The sign configuration $\epsilon^\pm_j$ is precisely what turns each one-point weight $\alpha_i$ in \eqref{sum1} into $(\theta_{\delta(i)}v)^{j_i}$, up to the sign $(-1)^{N-1}$ when $i\in I^-$, and thereby makes the row signs $\lambda_i$ in \eqref{lambda} depend on the row only through $(\theta_{\delta(i)}\sigma)^{j_i}$.
\end{remark}

\bibliographystyle{amsplain}

\end{document}